\documentclass[letterpaper,11pt]{article}

\usepackage[margin=1in]{geometry}  
\usepackage{xspace,exscale,relsize}
\usepackage{fancybox,shadow}
\usepackage{graphicx}
\usepackage{color}
\usepackage{amsthm,amsfonts}
\usepackage{amssymb}
\usepackage{authblk}
\usepackage[title]{appendix}

\usepackage{extarrows}

\usepackage[noadjust]{cite}

\usepackage{amsmath}
	\makeatletter
	\let\over=\@@over \let\overwithdelims=\@@overwithdelims
	\let\atop=\@@atop \let\atopwithdelims=\@@atopwithdelims
  	\let\above=\@@above \let\abovewithdelims=\@@abovewithdelims
  	\makeatother
\usepackage{rotating}

\usepackage{ifpdf}

\usepackage{subfigure}
\usepackage{psfrag}

\usepackage{prettyref}
\usepackage{enumitem}

\usepackage{tikz}
\usetikzlibrary{arrows}
\tikzstyle{int}=[draw, fill=blue!20, minimum size=2em]
\tikzstyle{dot}=[circle, draw, fill=blue!20, minimum size=2em]
\tikzstyle{init} = [pin edge={to-,thin,black}]

\usepackage[
            CJKbookmarks=true,
            bookmarksnumbered=true,
            bookmarksopen=true,
            colorlinks=true,
            citecolor=red,
            linkcolor=blue,
            anchorcolor=red,
            urlcolor=blue,
            pdfauthor={Polyanskiy and Wu}
            ]{hyperref}

\usepackage[all]{xy}

\usepackage{mathtools}

\usepackage{ifthen}
\newboolean{aos}
\setboolean{aos}{FALSE}

\newcommand\numberthis{\addtocounter{equation}{1}\tag{\theequation}}

\newcommand{\mreals}{\ensuremath{\mathbb{R}}}

\newcommand{\supp}{\ensuremath{\mathrm{supp}}}

\ifx\eqref\undefined
	\newcommand{\eqref}[1]{~(\ref{#1})}
\fi
\ifx\mod\undefined
	\def\mod{\mathop{\rm mod}}
\fi

\usepackage{bm}

\def\argmin{\mathop{\rm argmin}}

\def\exp{\mathop{\rm exp}}

\def\EE{\Expect}

\def\Var{\mathrm{Var}}

\def\PP{\mathbb{P}}

\def\eqdef{\triangleq}

\def\Leb{\mathrm{Leb}}
\def\simiid{\stackrel{iid}{\sim}}

\newcommand{\stepa}[1]{\overset{\rm (a)}{#1}}
\newcommand{\stepb}[1]{\overset{\rm (b)}{#1}}
\newcommand{\stepc}[1]{\overset{\rm (c)}{#1}}
\newcommand{\stepd}[1]{\overset{\rm (d)}{#1}}

\newcommand{\Poi}{\mathrm{Poi}}

\newcommand{\reals}{\mathbb{R}}

\newcommand{\integers}{\mathbb{Z}}

\newcommand{\Expect}{\mathbb{E}}
\newcommand{\expect}[1]{\mathbb{E}\left[#1\right]}
\newcommand{\Prob}{\mathbb{P}}
\newcommand{\prob}[1]{\mathbb{P}\left[#1\right]}
\newcommand{\pprob}[1]{\mathbb{P}[#1]}

\newcommand{\TV}{{\rm TV}}

\newcommand{\iid}{i.i.d.\xspace}

\newcommand{\pth}[1]{\left( #1 \right)}
\newcommand{\qth}[1]{\left[ #1 \right]}
\newcommand{\sth}[1]{\left\{ #1 \right\}}

\newcommand{\iiddistr}{{\stackrel{\text{\iid}}{\sim}}}

\newcommand{\Bern}{\text{Bern}}

\newcommand{\Bino}{\text{Binom}}

\newcommand{\indc}[1]{{1\left\{{#1}\right\}}}
\newcommand{\Indc}{\mathbf{1}}

\definecolor{myblue}{rgb}{.8, .8, 1}
\definecolor{mathblue}{rgb}{0.2472, 0.24, 0.6} 
\definecolor{mathred}{rgb}{0.6, 0.24, 0.442893}
\definecolor{mathyellow}{rgb}{0.6, 0.547014, 0.24}

\newcommand{\blue}{\color{blue}}
\newcommand{\nb}[1]{{\sf\blue[#1]}}

\newcommand{\calF}{{\mathcal{F}}}
\newcommand{\calG}{{\mathcal{G}}}

\newcommand{\calK}{{\mathcal{K}}}

\newcommand{\calM}{{\mathcal{M}}}
\newcommand{\calN}{{\mathcal{N}}}

\newcommand{\calP}{{\mathcal{P}}}

\newcommand{\calX}{{\mathcal{X}}}

\newcommand{\diverge}{\to \infty}
\newcommand{\mmse}{\mathsf{mmse}}

\newrefformat{eq}{(\ref{#1})}
\newrefformat{thm}{Theorem~\ref{#1}}
\newrefformat{th}{Theorem~\ref{#1}}
\newrefformat{chap}{Chapter~\ref{#1}}
\newrefformat{sec}{Section~\ref{#1}}
\newrefformat{seca}{Section~\ref{#1}}
\newrefformat{algo}{Algorithm~\ref{#1}}
\newrefformat{fig}{Fig.~\ref{#1}}
\newrefformat{tab}{Table~\ref{#1}}
\newrefformat{rmk}{Remark~\ref{#1}}
\newrefformat{clm}{Claim~\ref{#1}}
\newrefformat{def}{Definition~\ref{#1}}
\newrefformat{cor}{Corollary~\ref{#1}}
\newrefformat{lmm}{Lemma~\ref{#1}}
\newrefformat{prop}{Proposition~\ref{#1}}
\newrefformat{pr}{Proposition~\ref{#1}}
\newrefformat{app}{Appendix~\ref{#1}}
\newrefformat{apx}{Appendix~\ref{#1}}
\newrefformat{ex}{Example~\ref{#1}}
\newrefformat{exer}{Exercise~\ref{#1}}
\newrefformat{soln}{Solution~\ref{#1}}

\def\unifto{\mathop{{\mskip 3mu plus 2mu minus 1mu%
	\setbox0=\hbox{$\mathchar"3221$}%
	\raise.6ex\copy0\kern-\wd0%
	\lower0.5ex\hbox{$\mathchar"3221$}}\mskip 3mu plus 2mu minus 1mu}}

\ifx\lesssim\undefined
\def\simleq{{{\mskip 3mu plus 2mu minus 1mu%
	\setbox0=\hbox{$\mathchar"013C$}%
	\raise.2ex\copy0\kern-\wd0%
	\lower0.9ex\hbox{$\mathchar"0218$}}\mskip 3mu plus 2mu minus 1mu}}
\else
\def\simleq{\lesssim}
\fi

\ifx\gtrsim\undefined
\def\simgeq{{{\mskip 3mu plus 2mu minus 1mu%
	\setbox0=\hbox{$\mathchar"013E$}%
	\raise.2ex\copy0\kern-\wd0%
	\lower0.9ex\hbox{$\mathchar"0218$}}\mskip 3mu plus 2mu minus 1mu}}
\else
\def\simgeq{\gtrsim}
\fi

\newtheorem{theorem}{Theorem}
\newtheorem{lemma}[theorem]{Lemma}

\newtheorem{proposition}[theorem]{Proposition}
\newtheorem{prop}[theorem]{Proposition}

\theoremstyle{definition}

\newtheorem{remark}{Remark}

\newif\ifmapx
{\catcode`/=0 \catcode`\\=12/gdef/mkillslash\#1{#1}}
\edef\jobnametmp{\expandafter\string\csname embayes2_apx\endcsname}
\edef\jobnameapx{\expandafter\mkillslash\jobnametmp}
\edef\jobnameexpand{\jobname}
\ifx\jobnameexpand\jobnameapx
\mapxtrue
\else
\mapxfalse
\fi

\long\def\apxonly#1{\ifmapx{\color{blue}#1}\fi}

\newcommand{\bfH}{\mathbf{H}}

\newcommand{\poly}{\mathsf{poly}}

\newcommand{\Regret}{\mathsf{Regret}}

\newcommand{\TotRegret}{\mathsf{TotRegret}}
\newcommand{\AccRegret}{\mathsf{AccRegret}}
\newcommand{\TotRegretComp}{\mathsf{TotRegretComp}}
\newcommand{\SubG}{\mathsf{SubG}}
\newcommand{\SubE}{\mathsf{SubE}}

\newcommand{\oracle}{\mathsf{oracle}}

\renewcommand{\hat}{\widehat}
\renewcommand{\tilde}{\widetilde}

\def\mmse{\mathrm{mmse}}

\def\GammaD{\mathrm{Gamma}}
\begin{document}
\ifpdf
\DeclareGraphicsExtensions{.pgf}
\graphicspath{{figures/}{plots/}}
\fi

\title{Sharp regret bounds for empirical Bayes and compound decision problems}

\author{Yury Polyanskiy and Yihong Wu\thanks{Y.P. is with the Department of EECS, MIT, Cambridge,
MA, email: \url{yp@mit.edu}. Y.W. is with the Department of Statistics and Data Science, Yale
University, New Haven, CT, email: \url{yihong.wu@yale.edu}. Y.~Polyanskiy is
supported in part by the MIT-IBM Watson AI Lab, and the NSF Grants CCF-1717842, CCF-2131115. Y.~Wu is supported in part by the NSF
Grant CCF-1900507, NSF CAREER award CCF-1651588, and an Alfred Sloan fellowship. }}

\maketitle

\begin{abstract}


We consider the classical problems of estimating the mean of an
$n$-dimensional normally (with identity covariance matrix) or Poisson distributed vector under the squared loss.
In a Bayesian setting the optimal estimator is given by the (prior-dependent) conditional mean. In a
frequentist setting various shrinkage methods were developed over the last century. 
The framework of empirical Bayes, put forth by Robbins \cite{Robbins56}, combines Bayesian and frequentist mindsets by postulating that the  parameters are
independent but with an unknown prior and aims to use a fully data-driven estimator to compete with the Bayesian oracle that knows the true prior.
The central figure of merit is the regret, namely, the total excess risk over the Bayes risk in the worst case (over the priors).
Although this paradigm was introduced more than 60 years ago, little is known about the asymptotic scaling of the optimal regret in the nonparametric setting.


We show that for the Poisson model with compactly supported and subexponential priors, the optimal regret scales as $\Theta((\frac{\log n}{\log\log n})^2)$ and $\Theta(\log^3 n)$, respectively, both attained by the original estimator of Robbins.
For the normal mean model, the regret is shown to be at least $\Omega((\frac{\log n}{\log\log n})^2)$ and $\Omega(\log^2 n)$ for compactly supported and subgaussian priors, respectively, the former of which resolves the conjecture of Singh \cite{Singh79} on the impossibility of achieving bounded regret; 
before this work, the best regret lower bound was $\Omega(1)$.
In addition to the empirical Bayes setting, these results are shown to hold in the compound setting where the parameters are deterministic.
As a side application, the construction in this paper also leads to improved or new lower bounds
for density estimation of Gaussian and Poisson mixtures.
\end{abstract}

\tableofcontents

\section{Introduction}
	\label{sec:intro}
Consider estimating an unknown parameter $\theta \in \mreals^{10}$ from a single measurement $Y
\sim \mathcal{N}(\theta, I_{10})$. Suppose that the observed data is
$$ Y= [1.08,\, -2.43,\,  1.52,\, -1.17,\,  1.39,\,  8.3 ,\, 10.02,\, 10.61,\, 9.3 ,\, 10.14 ]\,.$$
Although the worst-case reasoning suggests simply using the maximum likelihood estimator (MLE) $\hat
\theta = Y$, it is well-known by now that it is often advantageous to use shrinkage estimators,
such as James-Stein. However, in the example shown,\footnote{For a bigger example in the same spirit, see Efron's ``two towers'' example \cite[Figure 1]{efron2019bayes}.} it should also be clear that applying the James-Stein estimator on the first 5 coordinates and the rest \textit{separately} would be even more advantageous since shrinking to a common mean of all 10
coordinates (in this bimodal situation) appears unnatural. How can one formalize derivation of such clever procedures and how
can one study their fundamental limits? Both of these questions were addressed in a decision-theoretic framework
proposed by Robbins~\cite{Robbins51,Robbins56} under the name of \textit{empirical Bayes}.
Since then a lot has been written about this paradigm and empirical Bayes methodology has been widely used in practice for large-scale data analysis, with notable applications in 
computational biology especially microarrays \cite{efron2001empirical}, sports prediction \cite{brown2008season}, etc.
We refer to the survey articles \cite{casella1985introduction,zhang2003compound,efron2021empirical} and the monograph \cite{efron2012large} for theory and methodology for empirical Bayes.

The basic tenet and the major surprise of the empirical Bayes theory is that, when the number of independent observations is large, it is possible to ``borrow strength'' from these independent (and seemingly unrelated) observations so as to achieve the optimal Bayes risk per coordinate asymptotically.
In this work we study the principal performance metric (regret) and its asymptotic scaling as the
number of observations grows, thereby characterizing the optimal speed to approach the Bayes risk. We proceed to formal definitions.

	\subsection{Definitions and main results}
	\label{sec:def}

Throughout the paper, we use the vector notation $\theta^n\triangleq (\theta_1,\ldots,\theta_n)$. 
Let $\{P_\theta: \theta\in \Theta\subset\reals\}$ be a parametric family of distributions 
 with a bi-measurable density $f_\theta(y)$ with respect to a common dominating measure $\nu$.
Conditioned on the parameters $\theta^n$, the observations $Y_1,\ldots,Y_n$ are independently distributed according to $P_{\theta_i}, i=1,\ldots,n$.


Following the terminology in \cite{Robbins51,Robbins56,zhang1997empirical,zhang2003compound}, we consider both the \emph{empirical Bayes} (EB) and \emph{compound} settings, in which 
$\theta_i$'s are iid (with an unknown prior) or deterministic, respectively.
In both settings, the goal is to estimate the parameters $\theta_1,\ldots,\theta_n$ under the quadratic loss and compete with the oracle that has extra distributional information (e.g.~the prior or the empirical distribution of the parameters).

%
%
	%

We start with the empirical Bayes setting. 
Let $\theta^n=(\theta_1,\ldots,\theta_n)\iiddistr G$. Then $Y_i \iiddistr P_G$, where $P_G = \int P_\theta G(d\theta)$ is the mixture under the prior (mixing distribution) $G$ with density $f_G(y) = \int f_\theta(y) G(d\theta)$.
Denote by $\Expect_{G}$ the expectation taken over the joint distribution of $(\theta^n,Y^n)$.
The Bayes risk under the prior $G$ is the minimum mean squared error (MMSE) of estimating $\theta\sim G$ based on a single observation $Y\sim f_\theta$:
\begin{equation}
\mmse(G) \triangleq \inf_{\hat\theta} \Expect_G[(\hat\theta(Y) - \theta)^2] = \Expect_G[(\hat\theta_G(Y) - \theta)^2],
\label{eq:mmse}
\end{equation}
where 
\begin{equation}
\hat\theta_G(y) \triangleq \Expect_G[\theta|Y=y] = \frac{\int_\Theta \theta f_\theta(y) G(d\theta)}{\int_\Theta f_\theta(y) G(d\theta)}
\label{eq:bayes}
\end{equation}
is the Bayes estimator (conditional mean) associated with the prior $G$.
In this paper we will be concerned with the following basic models:
\begin{itemize}
	\item Normal mean model: $P_\theta = \mathcal{N}(\theta, 1), \theta \in \mreals$, with density $f_\theta(y) = \varphi(y-\theta)$ and $\varphi(y) \triangleq \frac{1}{\sqrt{2\pi}} e^{-y^2/2}$. For any prior $G$, the corresponding Bayes estimator can be expressed using the mixture density $f_G$ and its derivative as (Tweedie's formula \cite{efron2011tweedie}):
	\begin{equation}
\hat\theta_G(y) = y + \frac{f_G'(y)}{f_G(y)},
\label{eq:bayes-gau}
\end{equation}
with the second term known as the Bayes correction.

	\item Poisson model: $P_\theta = \Poi(\theta), \theta \in \mreals_+$, with probability mass function $f_\theta(y) = \frac{ e^{-\theta} \theta^y}{y!}$ and $y \in \integers_+ \triangleq \{0,1,\ldots\}$. The Bayes estimator takes the form:
		\begin{equation}
\hat\theta_G(y) = (y +1) \frac{f_G(y+1)}{f_G(y)}.
\label{eq:bayes-poi}
\end{equation}
	
\end{itemize}
Both models are well studied and widely applied in practice. 
In addition to denoising, the normal mean model has also found usefulness in other high-dimensional problems such as principle component analysis. 

The main objective of empirical Bayes is to compete with the oracle that knows the true prior and minimize the excess risk, known as the \emph{regret} \cite{jiang2009general,efron2011tweedie,efron2019bayes}, by mimicking the true Bayes estimator.
This strategy is either carried out explicitly (by learning a prior from data then executing the learned Bayes estimator) or implicitly (by constructing an approximate Bayes estimator using the marginal distribution of the observation directly).\footnote{These two approaches are referred to as $g$-modeling and $f$-modeling in \cite{efron2014two}, respectively.}
 Robbins' famous estimator for the Poisson model \cite{Robbins56} belongs to the latter category. Namely, we estimate $\theta_j$ by 
\begin{equation}\label{eq:robbins}
	\tilde \theta_j \eqdef (Y_j+1) {N(Y_j+1)\over N(Y_j)}\,, \quad N(y)\eqdef
|\{i\in[n]:Y_i=y\}|.
\end{equation}
which can be viewed as a plugin version of the Bayes estimator \prettyref{eq:bayes-gau} replacing the marginal $f_G$ by the empirical version.
Despite the broad application of the Robbins estimator in practice, little is known about its optimality.

To study the fundamental limit of empirical Bayes estimation, let us define the optimal regret. Given a collection of priors $\calG$, the total regret is defined as
	\begin{equation}
	\TotRegret_n(\calG) \triangleq \inf_{\hat \theta^n} \sup_{G \in \calG} 	\sth{\Expect_G\qth{\|\hat \theta^n(Y_1,\ldots,Y_n) - \theta^n\|^2} - n \cdot \mmse(G)}\,.
	\label{eq:avgregret}
	\end{equation}
	where $\|\cdot\|$ denotes the Euclidean norm and the infimum is taken over all estimator $\hat\theta^n$ with respect to $Y_1,\ldots,Y_n$.
	In this paper we focus on \emph{nonparametric} settings; for example, the class $\calG$ consists of all priors over a compact parameter space.
	
	Under appropriate assumptions, Robbins demonstrated it is possible to achieve a \emph{sublinear} total regret $o(n)$, so that the excess risk per observation is amortized as the sample size $n$ tends to infinity. A natural question is to determine the order of the optimal regret.
	Singh conjectured \cite[p.~891]{Singh79} that for any exponential family,
bounded total regret is not possible even if the priors are compactly supported.
In this work we resolve this conjecture in the positive for both normal and Poisson model.
Before this work the best lower bound is $\Omega(1)$ \cite{li2005convergence}, which can be shown by considering a parametric family of priors.

Our main results are as follows:

\begin{theorem}[Normal means problem]
\label{th:gaussian} 
Consider the normal mean model: $P_\theta = \mathcal{N}(\theta, 1), \theta \in \mreals$.

\begin{itemize}
	\item 
Let $\calP([-h,h])$ be the set of all distributions supported on $[-h,h]$. 
There exist constants $c_0=c_0(h)>0$ and $n_0$, such that for all $n\geq n_0$,
	\begin{equation}\label{eq:gsn_lb2}
		\TotRegret_n(\calP([-h,h]) \ge c_0 \left({\log n \over \log \log n}\right)^2\,.
	\end{equation}		
	
	\item 
	Denote by $\SubG(s)$ the set of all $s$-subgaussian distributions, namely
	$\SubG(s) = \{G: G([-t,t]^c) \leq 2 e^{-t^2/(2s)}, \forall t>0\}$. 
	Then for some constant $c_1=c_1(s)>0$ and all $n\ge n_0$ 
	\begin{equation}\label{eq:gsn_lb1}
		\TotRegret_n(\SubG(s)) \ge c_1 \log^2 n\,.
	\end{equation}

	\end{itemize}
	
\end{theorem}

\begin{theorem}[Poisson problem]\label{th:poisson} 
Consider the Poisson model $P_\theta = \Poi(\theta), \theta \in \mreals_+$. 

\begin{itemize}
	\item 
	For $h>0$, there exist constants $c_1=c_1(h), c_2=c_2(h)$ such that for all $n\ge n_0$,
	\begin{equation}\label{eq:poi_lb}
		c_1 \left({\log n \over \log
		\log n}\right)^2\leq \TotRegret_n(\calP([0,h]) \leq c_2 \left({\log n \over \log
		\log n}\right)^2\,.
\end{equation}

\item Denote by $\SubE(s)$ the set of all $s$-subexponential distributions on $\reals_+$, namely
	$\SubE(s) = \{G: G([t,\infty)) \leq 2e^{-t/s}, \forall t>0\}$. 
	Then for some constants $c_3=c_3(s), c_4=c_4(s)$ and all $n\ge n_0$,
	\begin{equation}\label{eq:poi_lb2}
		c_3 \log^3 n \leq \TotRegret_n(\SubE(s)) \leq c_4 \log^3 n\,.
\end{equation}		

\end{itemize}
Furthermore, in both \prettyref{eq:poi_lb} and \prettyref{eq:poi_lb2}, the upper bound is achieved by the estimator \prettyref{eq:robbins} of Robbins.

\end{theorem}

%
	%
	%
		%


A few remarks are in order:
\begin{enumerate}
	\item 
	It is known that tail assumptions (e.g.~compact support	or moment conditions such as subgaussianity) are necessary for achieving non-trivial regret.
	In fact, in the normal means problem, if we allow all priors on $\reals$, \cite[Example 1]{zhang2009generalized} showed that $\TotRegret_n=n+o(n)$, which is almost as large as possible.

\item The upper bound in \prettyref{eq:poi_lb} for compactly supported priors is essentially contained in~\cite{BGR13}, who analyzed the risk of the Robbins estimator with sample size $n$ replaced by $\Poi(n)$, which renders the empirical counts $N(y)$'s in \prettyref{eq:robbins} independent for technical convenience.
In \prettyref{app:robbins} we bound the regret of the Robbins estimator without Poissonization following essentially the same strategy as in~\cite{BGR13} and fix a misstep therein.

\item
The notion of \emph{conditional regret} is defined in \cite[Sec.~4]{efron2011tweedie}, which refers to the excess prediction risk of a given EB estimator over the oracle (Bayes) risk conditioned on the value $y$ of the observation. For several parametric classes of priors and parametric EB estimators, such as Gaussian priors and the James-Stein estimator (see \cite[Eq.~(4.10)]{efron2011tweedie}), the conditional regret is shown to be $\frac{c(y)}{n}$ for some explicit constant $c(y)$. Our results show that for nonparametric classes of priors, 
this constant $c(y)$ is not uniformly bounded and the conditional regret in the worst case is of higher order than $\frac{1}{n}$ by logarithmic factors.

	\item In addition to Robbins estimator, which does not take a Bayes form, the optimal regret in \prettyref{eq:poi_lb} and \prettyref{eq:poi_lb2} in the Poisson model can also achieved by a more principled approach of first estimating the prior then using the corresponding Bayes estimator. 
As suggested by Robbins \cite[Section 6]{Robbins56}, one could use Wolfowitz's minimum-distance estimator \cite{Wolfowitz57}.
This strategy indeed succeeds for a class of minimum-distance estimators, including the nonparametric MLE (NPMLE) \cite{KW56,lindsay1983geometry1} which minimizes the Kullback-Leibler divergence \cite{JPW21}.
In contrast, for the Gaussian model, the current best regret upper bound is $O(\log^5 n)$ for the subgaussian (hence also compactly supported) class \cite{jiang2009general} -- see \prettyref{sec:related} for a more detailed discussion.
	

\item An interesting open question for the Poisson model is to obtain a constant-factor
characterization of regret uniformly over $h$ (or $s$) in $(0,\infty)$. At this point it is not
clear how magnitude of $h$ (or $s$) affects the regret and such characterization would clarify
this.
	
\end{enumerate}

\paragraph{Compound setting.} 
In the compound estimation problem, the unknown parameter $\theta^n$ is no longer assumed to be random. 
Again, we are interested in estimating this deterministic vector $\theta^n$ by competing with an oracle possessing some
side information about $\theta^n$. There are two natural choices of the oracle
resulting in the two definitions of the regret, that we define below. For a given $\theta^n$ let
us denote its empirical distribution by
\begin{equation}
G_{\theta^n}  \triangleq {1\over n} \sum_{i=1}^n \delta_{\theta_i}.
\label{eq:emp}
\end{equation}
 A simple observation is that
$$ \inf_{f_1} \sum_{i=1}^n \EE_{\theta_i}[(\theta_i - f_1(Y_i))^2] = n \cdot \mmse(G_{\theta^n})\,,$$
achieved by the Bayes estimator under the prior being the empirical distribution, namely, $\EE_{\theta \sim G_{\theta^n}}[\theta|Y]$.
This leads us to defining the (most commonly used) version of the regret in the compound setting:
\begin{equation}
\TotRegretComp_n(\Theta) \eqdef \inf_{\hat \theta^n} \sup_{\theta^n \in \Theta^{\otimes n}} \sth{ \EE[\|\hat
\theta^n - \theta^n\|^2] - n \cdot \mmse(G_{\theta^n})}\,,
\label{eq:totregretcomp1}
\end{equation}
which, by the observation above, corresponds to competing with an oracle who can adapt to $\theta^n$
but is restricted to use a separable estimator, namely, applying a common univariate rule $f_1$ to all coordinates.

An alternative notion of regret was introduced in~\cite{greenshtein2009asymptotic}, which can be
formalized by defining
$$ R_{\oracle}(G_{\theta^n}) = \inf_{f} \sup_{\sigma \in S_n} \EE_{\sigma \theta^n}[\|\sigma \theta^n -
f(Y^n)\|^2]\,,$$
where $S_n$ is the group of permutations on $[n]$ and $\sigma \theta^n \triangleq
(\theta_{\sigma(1)},\ldots,\theta_{\sigma(n)})$. One can show that the value of $R_{\oracle}$ is
unchanged if one (a) replaces $\sup_\sigma$ with average over a random uniform $\sigma$; or (b) if one removes
$\sup\sigma$ and instead restricts the infimum to permutation-invariant functions
$f:\mreals^n\to\mreals^n$. (As~\cite{greenshtein2009asymptotic} shows, the optimal $f$ is given by 
the conditional expectation $\EE[\theta^n|Y^n]$ under the uniform prior on the
$S_n$ orbit of the support of $G_{\theta^n}$.) The corresponding notion of regret is then defined as\footnote{Conditions under which $\TotRegretComp_n' = \TotRegretComp_n + O(1)$ are given in \cite[Theorem 5.1]{greenshtein2009asymptotic}, which are satisfied by bounded normal mean model 
\cite[Corollary 5.2]{greenshtein2009asymptotic}.}
\begin{equation}
 \TotRegretComp_n'(\Theta) \eqdef \inf_{\hat \theta^n} \sup_{\theta^n \in \Theta^{\otimes n}} \sth{ \EE[\|\hat
\theta^n - \theta^n\|^2] - R_{\oracle}(G_{\theta^n})}\,,
\label{eq:totregretcomp2}
\end{equation}
which can be understood either as (a) competing with the oracle that has information about $G_{\theta^n}$ (but
not $\theta^n$), or (b) can adapt to $\theta^n$ but is restricted to use permutation-invariant
estimators $\hat \theta^n$.

It turns out that our lower bounds for the EB regret also apply to both versions of the compound regret.
Indeed, denote by $\calP(\Theta)$ the collection of all priors on the parameter space $\Theta$. The following general lemma relates 
the regret in the compound setting to that of the EB setting (see \prettyref{app:comp} for a proof).
As a result, for bounded parameter space, 
the $\Omega((\frac{\log n}{\log\log n})^2)$ regret lower bound in \prettyref{th:gaussian} and \prettyref{th:poisson} 
also hold in the compound setting.
\begin{proposition}
\label{prop:regret-comp}	
	\begin{equation}\label{eq:cc_1}
	\TotRegretComp_n'(\Theta) \ge \TotRegretComp_n(\Theta) \ge \TotRegret_n(\calP(\Theta))\,.
\end{equation}
\end{proposition}

Extending the definitions in \prettyref{eq:totregretcomp1} and \prettyref{eq:totregretcomp2}, one can define the corresponding total regret in the constrained setting, 
denoted by $\TotRegretComp_n(\calG)$ and $\TotRegretComp_n'(\calG)$ respectively, where the supremum is taken over $\theta^n\in\reals^n$ whose empirical distribution belongs to a prescribed class $\calG$. 
Similar to \prettyref{prop:regret-comp}, one can relate the regret in the compound setting to that of the EB setting.
The following result (also proved in \prettyref{app:comp}) provides such a reduction for the normal means problem with subgaussian priors, which shows the $\Omega(\log^2 n)$ lower bound in \prettyref{th:gaussian} continues to hold in the compound setting.

\begin{proposition}
\label{prop:comp-subg}
Consider the normal means problem. There exists a universal constant $c_0$ such that for any $s>0$,
\[
\TotRegretComp(\SubG(s)) \geq \TotRegret(\SubG(c_0 s)) - \frac{s}{c_0 n},
\]
where $\SubG(s)$ denotes the collection of $s$-subgaussian distributions in \prettyref{th:gaussian}.
\end{proposition}

\apxonly{
Finally, we mention that~\cite{greenshtein2009asymptotic} gives conditions under which
$\TotRegretComp_n' \le \TotRegretComp_n + O(1)$ (satisfied by the normal means problem and
$\Theta$--compact \cite[Cor.~5.2]{greenshtein2009asymptotic}). It would be natural to conjecture that also $\TotRegretComp_n \le \TotRegret_n
+ O(1)$. A possible way to show this would be to compare, for a general permutation-invariant
estimator $f:\mreals^n \to \Theta^n$ the following two quantities: For any empirical distribution $\hat
G_n$ let $U(G_n)$ be the uniform distribution over all $\theta^n$'s with a given $\hat
G_n$, then 
$$ \EE_{\theta^n \sim U(G_n)}[\|\sigma\theta^n
- f(Y^n)\|^2]  = \EE_{\theta^n \simiid G_n}[\|\theta^n
- f(Y^n)\|^2 + o(n) \qquad ? $$
Unfortunately, in general even for compact $\Theta$ these two can differ by $\Omega(n)$. 
Consider the normal means problem with $\Theta=[-1,1]$ and let $G_n = {1\over 2}
\delta_{-1} + {1\over 2} \delta_{+1}$ (for $n$--even). Then if $\theta^n \sim U(G_n)$ we
have
$$ (\theta_1, Y_1, {1\over \sqrt{n}} \sum_{i} Y_i) \stackrel{d}{\to} (A,B, Z)\,,$$
where $A$ is uniform on $\pm 1$ and given $A$ we have $B\sim \mathcal{N}(A,1)$ and $Z\sim
\mathcal(0,1)$ independent of $(A,B)$. At the same time under $\theta^n \simiid G_n$ we have 
$$ (\theta_1, Y_1, {1\over \sqrt{n}} \sum_{i} Y_i) \stackrel{d}{\to} (A,B,\sqrt{2}
Z)\,,$$
where $\sigma_2 \neq \sigma_1$. Thus, if we set $\hat
\theta_j = t(Y_j, {1\over \sqrt{n}}\sum_i Y_i)$ and $t(y, s) = \mathrm{sign y} 1\{s>1\}$ then 
$$ \EE_{\theta^n \sim U(G_n)}[\|\theta^n - \hat \theta^n\|^2] - \EE_{\theta^n \simiid \hat
G_n}[\|\theta^n - \hat \theta^n\|^2] = \Omega(n)\,.$$
This shows that a generic reduction $\TotRegretComp_n \le \TotRegret_n + o(n)$ is not possible.}

\subsection{Related work}
\label{sec:related}

The field of empirical Bayes and compound decision is exceedingly rich in terms of both theoretical and methodological development. 
In this paper we have focused on the specific problem of \emph{estimation} under the squared error
in both the normal and Poisson model, with the main objective of determining the optimal regret for various nonparametric classes of priors.
Next we review the existing literature that are directly related to our main results. Later in \prettyref{sec:discuss} we discuss a few different formulations and perspectives in empirical Bayes in connection with the present paper.

As mentioned earlier, the best regret lower bound before this work is $\Omega(1)$ \cite[Theorem 2.2]{li2005convergence}, which is shown by a two-point argument (Le Cam's method), namely, choosing two priors which differ by $\frac{1}{\sqrt{n}}$ in mean and whose corresponding mixture distributions has statistical distance (e.g.~KL divergence) $O(1/n)$.
As well-known in the high-dimensional statistics literature, two-point method often fails to capture the correct dependency on the model complexity. Indeed, in order to show superconstant regret, it is necessary to consider many hypotheses and Assouad's lemma provides a principled way to do so; see \prettyref{sec:general} for a general program of proving regret lower bound.

For regret upper bound in the Poisson model, the only quantitative result appears to be \cite{BGR13} analyzing the original Robbins estimator with Poissonized sample size, which turns out to be optimal in view of \prettyref{th:poisson}. For the normal mean model, a lot more is known but none meets the lower bound in \prettyref{th:gaussian}.
The state of the art is achieved by the Generalized Maximum Likelihood Empirical Bayes (GMLEB) method in \cite{jiang2009general}. 
	The idea of GMLEB is to first estimate a prior using the NPMLE, then using the (regularized) Bayes rule with this estimated prior.
	Various regret bounds are obtained under moment or subgaussian assumptions on the prior by a reduction from the Hellinger risk of density estimation to regret \cite[Theorem 3 (ii)]{jiang2009general}, and leveraging density estimation results for Gaussian mixtures \cite{ghosal.vdv,zhang2009generalized}. In particular, for the subgaussian class, it is shown that 
	$\TotRegret(\SubG(C)) = O(\log^5 n)$ for constant $C$. (In comparison, \prettyref{th:gaussian} shows $\TotRegret(\SubG(C)) = \Omega(\log^2 n)$.)
	Furthermore, these upper bounds on the regret (additive error) are extended to the compound setting \cite[Theorems 1 and 5]{jiang2009general} leading to attaining the Bayes risk (under the empirical distribution as the prior) up to multiplicative factor of $1+o(1)$, referred 
to as \emph{adaptive ratio optimality}, provided that the Bayes risk is at least $\Omega(\frac{\log^5 n}{n})$.
	\apxonly{Note that in the compound setting \cite{jiang2009general} does not produce an additive upper bound like $\TotRegretComp(\SubG(C)) = O(\log^5 n)$, because the }


In addition to GMLEB, kernel-based EB estimator is also commonly used. 
Note that unlike the discrete Poisson model, there exists no unbiased estimator for the mixture density, so there is no counterpart of Robbins estimator \prettyref{eq:robbins} for the Gaussian model. Nevertheless, the form of the Bayes estimator \prettyref{eq:bayes-gau} motivates using kernel-based estimators in place of the mixture density and its derivative.
Using polynomial kernels, Singh \cite{Singh79} obtained upper bound for general exponential families, which, for the bounded normal means problem, leads to $n^{o(1)}$ upper bound for the total regret, which is subsequently improved to $O((\log n)^8)$ in \cite[Example 1]{li2005convergence} by using polynomials of degree that scales logarithmically in $n$. 
On the other hand, in view of the analyticity of the Gaussian mixture, \cite{zhang2009generalized} uses sinc kernel to achieve adaptive ratio optimality and a total regret that scales as $\sqrt{n}$. Similar result is also obtained in \cite{brown2009nonparametric} using Gaussian kernel under stronger assumptions.

Another line of research aims to obtain regret bound under regularity assumptions on the prior, e.g., by imposing conditions on the Fourier transform $\tilde g(\omega)$ of the prior density $g$. Pensky \cite{pensky1999nonparametric} constructed a wavelet-based EB estimator $\tilde \theta(\cdot)=\tilde \theta(\cdot;Y_1,\ldots,Y_n)$, which, under the assumption that $\int |\tilde g(\omega)|^2 (\omega^2+1)^\alpha < \infty$ for some $\alpha>0$, satisfies $\Expect[(\tilde \theta(y)-\theta_G(y))^2] = O(\frac{(\log n)^{3/2}}{n})$ uniformly for $y$ in any compact interval. For the related testing problem, Liang \cite{liang2004optimal} showed that if $\tilde g$ has bounded support then the parametric rate $O(\frac{1}{n})$ is achievable by a kernel-based EB estimator. In this paper, we only impose tail conditions on the prior rather than smoothness of the prior density.


	\subsection{Notations}
	\label{sec:notations}
	
	Given a measurable space $(\calX,\calF)$, let $\calP(\calX)$ denote the collection of all probability measures on $(\calX,\calF)$.
	Throughout this paper, $\calX$ will be taken to be a subset of $\reals$ and $\calF$ the Borel $\sigma$-algebra.	
	The Lebesgue measure on $\reals$ is denoted by $\Leb$.
	Given a measure $\mu$ on $\calX$, let $L_2(\mu)$ 
denote the collection of all square-integrable functions $f: \calX \to\reals$ with $\|f\|_{L_2(\mu)} \triangleq (\int_{\calX} f(x)^2 \mu(dx))^{1/2}$ and 
inner product $(f,g)_{L_2(\mu)} \triangleq \int_{\calX} f(x)g(x) \mu(dx)$.
	
	Throughout the paper we adopt the vector notation $y^n\triangleq (y_1,\ldots,y_n)$ which applies to both deterministic and random vectors.	
	We use standard asymptotic notations: For any sequences $\{a_n\}$ and $\{b_n\}$ of positive numbers, we write $a_n \gtrsim b_n$ if $a_n\geq cb_n$ holds for all $n$ and some absolute constant $c > 0$, $a_n\lesssim b_n$ if $a_n \gtrsim b_n$, and $a_n \asymp b_n$ if both $a_n\gtrsim b_n$ and $a_n\lesssim b_n$ hold; the notations $O(\cdot)$, $\Omega(\cdot)$, and $\Theta(\cdot)$ are similarly defined. We write $a_n=o(b_n)$ or $b_n=\omega(b_n)$ or $a_n\ll b_n$ or $b_n \gg a_n$ if $a_n/b_n\to 0$ as $n\diverge$.
	
	\subsection{Organization}
	\label{sec:org}
	
	The rest of the paper is organized as follows.
	\prettyref{sec:basic} discusses the basic reduction for EB regret and introduces a general program for proving regret lower bounds.
	In \prettyref{sec:pf} we apply this program to show the lower bounds in Theorems \ref{th:gaussian} and \ref{th:poisson} for both the Gaussian and Poisson model, making use of the analytical tools developed in Appendices \ref{apx:normal} and \ref{apx:poisson}.
\prettyref{sec:discuss} concludes the paper with a discussion of related problems and open questions.
Regret upper bound for Robbins estimator is shown in \prettyref{app:robbins}.
With most of the paper confusing on the EB setting, proofs for the compound setting is deferred to \prettyref{app:comp}.
As a side application, the construction in the present paper can be used to obtain improved or new
lower bounds for \textit{mixture density estimation}; \prettyref{app:density} collects these results.

\section{Preliminaries}
	\label{sec:basic}
	
	\subsection{Reduction to estimating regression functions}
	\label{sec:reduction}
	
	In this subsection we describe the basic reduction relating the total regret~\eqref{eq:avgregret} in the empirical Bayes setting
	to the problem of estimating a regression function, the latter of which is closely related to the ``density estimation'' problem in the Gaussian or Poisson mixture model. 
	In turn, the rest of our paper proceeds in the spirit of proving the lower bounds on the density estimation, namely by building up a
	large collection of densities and applying the Assouad's lemma to show estimation lower bound; cf. Section~\ref{sec:general} for details.

First, in addition to the total regret~\eqref{eq:avgregret} we define the individual regret.
	\begin{equation}
\Regret_n(\calG) \triangleq
\inf_{\hat \theta} \sup_{G \in \calG} \sth{ \Expect_G\qth{(\hat \theta(Y_1,\ldots,Y_n) - \theta_n)^2  } - \mmse(G)} \,.
\label{eq:regret}
\end{equation}
Here $Y_1,\ldots,Y_{n-1}$ can be viewed as training data from which we learn an estimator and apply it on a fresh (unseen) data point $Y_n$ in order to predict $\theta_n$.
For example, for the Robbins estimator $\{\tilde \theta_j\}$ in \prettyref{eq:robbins}, the learned estimator $\hat\theta_{\rm Robbins}(\cdot)=\hat\theta_{\rm Robbins}(\cdot;Y_1,\ldots,Y_{n-1})$ is given by
\[
\hat\theta_{\rm Robbins}(y) = (y+1) \frac{N_{n-1}(y+1)}{N_{n-1}(y)+1}
\]
where $N_{n-1}(y) = \sum_{i=1}^{n-1} \indc{Y_i=y}$ is the number of occurrences of $y$ in the first $n-1$ observations, so that $\tilde \theta_n = \hat\theta_{\rm Robbins}(Y_n;Y_1,\ldots,Y_{n-1})$.
 Note that the extra ``add one'' in the denominator is reminiscent of the Laplace estimator (cf.~\cite[Sec.~13.2]{cover}); this additional smoothing ensures a nonzero denominator which is clearly needed for achieving a non-trivial regret.
Similar regularization has also been applied in the Gaussian model to the maximum likelihood method \cite{jiang2009general}.

There is a simple relation between $\TotRegret$ and $\Regret$:
	\begin{lemma}
\label{lmm:avgregret}	
	\[
	\Regret_n(\calG) = \frac{1}{n} \TotRegret_n(\calG).
	\]
\end{lemma}
\begin{proof}
	Given any $\hat\theta$ for \prettyref{eq:regret}, define $\tilde\theta^n=(\tilde\theta_1,\ldots,\tilde\theta_n)$ for \prettyref{eq:avgregret} by
	$\tilde\theta_i=\hat\theta(Y_{\backslash i},Y_i)$. By symmetry, $\Expect[(\tilde\theta_i-\theta_i)^2] = \Expect[(\hat\theta-\theta_n)^2]$ for all $i$. 
	This shows that $\TotRegret_n(\calG)  \leq n \cdot \Regret_n(\calG)$.
	
	The other direction is less obvious. 
	Given any $\hat\theta^n=(\hat\theta_1,\ldots,\hat\theta_n)$ for \prettyref{eq:avgregret}, 
	define a \emph{randomized} estimator $\tilde\theta$ for \prettyref{eq:regret} as follows by setting 
	$\tilde\theta = \hat\theta_i(Y_1,\ldots,Y_{i-1},Y_n,Y_{i+1},\ldots,Y_{n-1},Y_i)$ with probability $\frac{1}{n}$ for each $i\in[n]$.
	In other words, we randomly select $i$ uniformly at random from $[n]$, then swap $Y_i$ and $Y_n$ in the sample and apply the strategy $\hat\theta_i$.
	Then by exchangeability,
	\[
	\Expect[(\tilde\theta-\theta_n)^2] = \frac{1}{n} \sum_{i=1}^n \Expect[(\hat\theta_i(Y_1,\ldots,Y_n,\ldots,Y_i)-Y_n)^2] = 
	\frac{1}{n} \sum_{i=1}^n \Expect[(\hat\theta_i(Y_1,\ldots,Y_n)-Y_i)^2] = \frac{1}{n} \Expect[\|\hat\theta^n-\theta^n\|^2].
	\]
	This shows that $\Regret_n(\calG) \leq \frac{1}{n} \TotRegret_n(\calG)$.
\end{proof}

By the orthogonality principle, the average quadratic risk of a given estimator $\hat\theta$ can be decomposed as 
$$\Expect_G[(\hat\theta-\theta)^2] = \mmse(G) + \Expect_G[(\hat\theta-\hat\theta_G)^2]\,.$$
Thus the regret can be viewed the quadratic loss of estimating the Bayes rule as a function, with the extra twist that the error is averaged with respect to the mixture distribution $P_G$ that depends on the underlying $G$. In other words,
\begin{align}
\Regret_n(\calG) 
= & ~ \inf_{\hat \theta} \sup_{G \in \calG} \Expect_G\qth{\pth{\hat \theta(Y_1,\ldots,Y_n) - \hat\theta_G(Y_n)}^2  } \label{eq:regret1}	\\
= & ~ \inf_{\hat \theta} \sup_{G \in \calG} \Expect_G\qth{\|\hat \theta - \hat\theta_G\|^2_{L_2(f_G)}  } \label{eq:regret2}	
\end{align}
where in \prettyref{eq:regret2} $\hat \theta(\cdot)$ is understood as 
$\hat \theta(Y_1,\ldots,Y_{n-1},\cdot)$.
Note a conceptual shift: with these reductions we can think of ${1\over n} \TotRegret_n$ as a
measure of the quality of estimation of the regression function after receiving $n-1$ training
samples.

\subsection{Truncation of priors}
\label{sec:truncate}

Frequently we are interested in priors with compact support, which can be unwieldy to use in the lower bound construction.
Next we provide a general truncation lemma that translates regret lower bound for priors with uniform tail bound (e.g.~subgaussian or subexponential) to that for compactly supported priors:

\begin{lemma}
\label{lmm:truncation}		
	Given $a>0$, let $\calG=\calP([-a,a])$ denote all probability measures supported on the interval $[-a,a]$.
	Let $\calG'$ be a collection of priors on $\reals$, such that $\sup_{G\in\calG'} G([-a,a]^c) \leq \epsilon \leq \frac{1}{2}$ for some $\epsilon=\epsilon(a)$, and 
	$\sup_{G\in\calG'} \Expect_{G}[\theta^4] \leq M$.
	Then
	\begin{equation}
	\Regret_n(\calG) \geq \Regret_n(\calG')  - 6 \sqrt{(M+a^4) n \epsilon}.
	\label{eq:truncation}
	\end{equation}	
\end{lemma}
\begin{proof}
Define $E\eqdef \{|\theta_i|\leq a, i=1,\ldots,n\}$ and notice that for any estimator $\hat
\theta$ taking values in $[-a,a]$ and any prior $G \in \calG'$ we have
	\begin{align}
	\Expect_{G}[(\hat\theta-\theta_n)^2]
	= & ~ \Expect_{G}[(\hat\theta-\theta_n)^2 \Indc_{E}] + \Expect_{G}[(\hat\theta-\theta_n)^2
	\Indc_{E^c}]   \nonumber\\
	\leq & ~ 	\Expect_{G}[(\hat\theta-\theta_n)^2 \mid E] +
	\sqrt{\Expect_{G}[(\hat\theta-\theta_n)^4] \cdot \Prob_G[E^c]}   \nonumber\\
	\leq & ~ 	\Expect_{G}[(\hat\theta-\theta_n)^2 \mid E] + \sqrt{8(M+a^4) \cdot n
	\epsilon}. \label{eq:tr_1}
	\end{align}
For any $G$, let $G_a$ denote its restriction on $[-a,a]$, i.e., $G_a(A) = \frac{G(A\cap[-a,a])}{G([-a,a])}$.
Then, we get the following chain of inequalities:
	\begin{align*}
	\Regret_n(\calG)
	= & ~ \inf_{|\hat \theta|\leq a} \sup_{G \in \calG} \Expect_G\qth{(\hat \theta(Y_1,\ldots,Y_n) - \theta_n)^2  } - \mmse(G)  \\
	\geq & ~ \inf_{|\hat \theta|\leq a} \sup_{G \in \calG'} \Expect_{G_a}\qth{(\hat \theta(Y_1,\ldots,Y_n) - \theta_n)^2  } - \mmse(G_a)  \\
	= & ~ \inf_{|\hat \theta|\leq a} \sup_{G \in \calG'} \Expect_{G} \qth{(\hat
	\theta(Y_1,\ldots,Y_n) - \theta_n)^2 \Big| E } - \mmse(G_a)  \\
	\stepa{\geq} & ~ \inf_{|\hat \theta|\leq a} \sup_{G \in \calG'} \Expect_{G} \qth{(\hat
	\theta(Y_1,\ldots,Y_n) - \theta_n)^2 \Big| E }	- \frac{1}{1-\epsilon}\mmse(G)  \\
	\stepb{\geq} & ~ \inf_{|\hat \theta|\leq a} \sup_{G \in \calG'} \Expect_{G} \qth{(\hat \theta(Y_1,\ldots,Y_n) - \theta_n)^2} - \sqrt{8(M+a^4) (1-(1-\epsilon)^n)} 	- \frac{1}{1-\epsilon}\mmse(G)  \\
	\stepc{\geq} & ~  \Regret_n(\calG') - \sqrt{8(M+a^4) n \epsilon} - 2 \epsilon \sqrt{M},
	\end{align*}
	where (a) follows from the simple fact (cf.~e.g.~\cite[Lemma
	2]{mmse.functional.IT})
	\begin{equation}
	\mmse(G_a) \leq \frac{1}{G([-a,a])} \mmse(G);
	\label{eq:mmse-truncate}
	\end{equation}
	(b) is from~\eqref{eq:tr_1};
	(c) follows from $\mmse(G) \leq \Expect_G[\theta^2] \leq \sqrt{M}$ and
	${1\over 1-\epsilon} \le 1 + 2\epsilon$ due to $\epsilon \le 1/2$.
\end{proof}

\subsection{General program for regret lower bound}
\label{sec:general}

We describe our general method for proving regret lower bounds next. 
Select a prior $G_0$ and let $f_0 \equiv f_{G_0} = \int f_\theta G_0(d\theta)$ denote the induced mixture density.
We aim to prove regret bound by restricting to priors whose likelihood ratio relative to $G_0$ is bounded from above and below.
Specifically, we will construct a collection of perturbations around $G_0$ whose induced perturbations on the regression function satisfy an approximate orthogonality property as required by Assouad's lemma.
To describe this program, let us first introduce the necessary notations.

Define an operator $K$ that maps a bounded measurable function $r$ to $Kr$ given by
\begin{equation}\label{eq:k_def}
	Kr(y) \eqdef \EE_{G_0}[r(\theta) | Y=y] = \frac{\int r(\theta) f_\theta(y) G_0(d\theta)}{f_0(y)}\,.
\end{equation}
In other words, $K$ is an integral operator with kernel function given by the posterior density.
Notice that for the prior $G$ defined by $dG = r(\theta) dG_0$, the corresponding mixture density is given by
\begin{equation}
f_G(y) = f_0(y) \cdot Kr(y)\,.
\label{eq:fpi-K}
\end{equation} 

Next, fix an arbitrary bounded function $r$ and for sufficiently small $\delta$, define the distribution $G_\delta$ by
	$$ d G_\delta \triangleq {1\over 1+\delta \int r dG_0} (1+\delta r) dG_0\,.$$
Denote by 
\begin{equation}
K\theta(y) = \EE_{G_0}[\theta|Y=y]
\label{eq:Kid}
\end{equation}
to be the result of applying $K$ to the identity function $\theta\mapsto \theta$. 
By the Bayes rule and using \prettyref{eq:fpi-K}, the regression function corresponding to $G_\delta$ satisfies:
\begin{align}
		\hat\theta_{G_\delta}(y)\equiv
		\EE_{G_\delta} [\theta|Y=y] &= \EE_{G_0}\left[ \theta {1+\delta r(\theta)\over 1+\delta Kr(y)}
		\middle | Y=y\right]\nonumber\\
		& = {K\theta + \delta K(\theta r)\over 1+\delta Kr} (y) \nonumber\\
		& = K\theta (y) + \delta {K(\theta r) - (K\theta) (Kr) \over 1+\delta Kr} (y) \nonumber\\
	& = K\theta(y) + \delta K_1 r(y) + \delta^2 {1\over 1+\delta K r(y)} (Kr)(y) \cdot (K_1 r)(y)\,, \label{eq:tu_kk1}
\end{align}
where we introduced an operator 
\begin{equation}\label{eq:k1_def}
	K_1 r \triangleq K(\theta r) - (K\theta) (Kr)\,.
\end{equation}
Note that the mapping from $r$ to the regression function $\hat\theta_{G_\delta}$ is nonlinear, but can be expressed in terms of the linear operators $K$ and $K_1$; these two operators play crucial roles in the lower bound construction.
In applications below it will often be that $K$ is a convolution-like operator, while $K_1$ will
be a differentiation operator (followed by convolution). The emergence of differentiation may
appear surprising at first, so it may be satisfying to notice from \prettyref{eq:tu_kk1}
an alternative definition of $K_1$ emphasizing the differentiation explicitly:
$$ K_1 r(y) = \left.{d\over d\delta}\right|_{\delta=0} \EE_{G_\delta} [\theta|Y=y]\,.$$


The main method for proving our lower bounds is contained in the following proposition.
We note that similar strategy based on Assouad's lemma has been used before in \cite{kim2014minimax,kim2020minimax} to show lower bound for estimating Gaussian mixture density. 
	Here the construction is more challenging in view of the non-linear relationship between the regression function and the mixture density (see \prettyref{eq:bayes-gau} and \prettyref{eq:bayes-poi} for example) which requires working with the kernel $K$ and $K_1$ defined above. 
	The dependency of the regression function on ``high-order'' information of the mixture density (such as its derivatives) suggests the rate for individual regret may be strictly slower than the density estimation rate. This turns out to be the case for the Poisson model \cite{JPW21}.

\begin{proposition}\label{prop:ass_gen} Fix a prior distribution $G_0$, constants $a,\tau,\tau_1,\tau_2,\gamma \ge 0$
and $m$ real-valued functions $r_1,\ldots,r_m$ on $\Theta$ with the following properties.
\begin{enumerate}
\item For each $i$, $ \|r_i\|_{\infty} \le a $;

\item For each $i$, $\|K r_i\|_{L_2(f_0)} \le \sqrt{\gamma}$;

\item For any $v \in \{0, \pm 1\}^m$,
	\begin{equation}\label{eq:rg_ag}
		\left\|\sum_{i=1}^m v_i K_1 r_i\right\|^2_{L_2(f_0)} \ge \tau \|v\|_2^2 - \tau_2 
\end{equation}	
\item For any $v\in \{0,1\}^m$, $\|\sum_{i=1}^m v_i K_1 r_i\|^2_{L_2(f_0)} \le \tau_1^2 m $;

\end{enumerate}
Then the regret over the class of priors $\calG = \{G: \left|{dG/dG_0} - 1\right| \le {1\over 2}\}$ satisfies
\begin{equation}\label{eq:propass_g}
	\Regret_n(\calG) \ge C \delta^2 (m(4\tau - \tau_1^2) - \tau_2)\,, \qquad \delta \triangleq \frac{1}{\max(\sqrt{n \gamma},
ma)}\,,
\end{equation}
where $C>0$ is an absolute constant.

\end{proposition}

\begin{remark} Note that the lower bound in \prettyref{eq:propass_g} is invariant to rescaling all $r_i$'s by a common factor. Thus in applications, it
is convenient to fix one of the parameters. For example, we will always choose $\tau=1$. In fact, we will select $r_i$'s
so that they are \emph{orthonormal} after the action of the kernel $K_1$, i.e. 
\[
(K_1 r_i, K_1 r_j)_{L_2(f_0)} = 1\{i\neq j\},
\]
 so that $\tau_1=1$ and $\tau_2=0$. In this case, the regret lower bound reads 
\begin{equation}\label{eq:ap1}
	\Regret_n \gtrsim \min\sth{\frac{m}{n\gamma}, \frac{1}{ma^2}}\,,
\end{equation} which, in later applications, further simplifies to 
$\Regret_n \gtrsim \frac{m}{n\gamma}$.
\end{remark}

\begin{proof}
Throughout the proof, $C_0,C_1,\ldots$ denote absolute constants.
Let $\mu_i \triangleq \int r_i dG_0$, which satisfies $|\mu_i| \le a$ by the first assumption. 
For each binary string $u=(u_1,\ldots,u_m) \in \{0,1\}^m$, denote
\begin{equation}
r_u \triangleq \sum_{i=1}^m u_i r_i, \quad h_u \triangleq Kr_u, \quad \mu_u \triangleq \sum_{i=1}^m u_i \mu_i.
\label{eq:rh}
\end{equation}
We define $2^m$ distributions indexed by $u\in \{0,1\}^m$ as
	\begin{equation}\label{eq:rg_c0g}
		dG_u \triangleq \frac{1+\delta r_u}{1+\delta \mu_u} dG_0,
\end{equation}	
where $\delta>0$ is chosen so that
	\begin{equation}\label{eq:rg_c1g}
		\delta m a \le {1\over 16}\,.
\end{equation}
Then
\[
f_u \triangleq {(1+\delta h_u) f_0\over 1+\delta \mu_u}.
\]
is the mixture density induced by the prior $G_u$. 
Let $\tilde \calG = \{G_u: u\in \{0,1\}^m\}$. 
By \prettyref{eq:rg_c1g}, each $G_u$ satisfies $\frac{1}{2}\leq\frac{dG_u}{dG_0} \leq \frac{3}{2}$ and hence $\tilde \calG \subset \calG$.
Abbreviate $T_u(y) \triangleq \EE_{G_u}[\theta|Y=y]$.
Then we have:
\begin{align}
\Regret_n(\calG) \geq 
\Regret_n(\tilde \calG) 
\stepa{=} & ~ \inf_{\hat T} \sup_{u \in \{0,1\}^m} \Expect_{G_u} \qth{\|\hat T - T_u\|_{L^2(f_u)}^2}	\nonumber \\
\stepb{\geq} & ~ \frac{1}{2} \inf_{\hat T} \sup_{u \in \{0,1\}^m} \Expect_{G_u} \qth{\|\hat T - T_u\|_{L^2(f_0)}^2}	\nonumber \\
\stepc{\geq} & ~ 	\frac{1}{8} \inf_{\hat u \in \{0,1\}^m} \sup_{u \in \{0,1\}^m} \Expect_{G_u} \qth{\|T_{\hat u} - T_u\|_{L^2(f_0)}^2} \label{eq:regret3ag}
\end{align}
where (a) follows from the representation \prettyref{eq:regret2} for individual regret; 
in (b) we used the fact that ${f_u\over f_0} \ge {1\over 2}$ to change from $L_2(f_u)$ to $L_2(f_0)$; 
(c) follows from the usual argument for restricting to proper estimators, namely, for any $\hat T$,
$\hat u=\argmin_{v\in \{0,1\}^m} \|\hat T - T_v\|_{L^2(f_0)}$ satisfies that, for any $u$, $\|\hat T - T_u\|_{L^2(f_0)} \leq 2 
\|T_{\hat u} - T_u\|_{L^2(f_0)}$.
In the remainder of this proof all $\|\cdot\|_2$ norms are with respect to~$L_2(f_0)$.

In order to apply Assouad's lemma, we next show that we have approximate orthogonality, i.e. for some $\epsilon$ and $\epsilon_1$ (to be specified) and all $u,v
\in \{0,1\}^m$ we have
\begin{equation}\label{eq:rg_1g}
		\|T_u - T_v\|^2_2 \ge \epsilon d_{\rm H}(u,v) - \epsilon_1,
\end{equation}	
where $d_{\rm H}(u,v)\triangleq \sum |u_i-v_i|$ stands for the Hamming distance.
Indeed, by \prettyref{eq:tu_kk1}, the regression function $T_u$ is given by
\begin{equation}
T_u = K\theta + \delta K_1 r + \delta^2 {h_u\over 1+\delta h_u}  K_1 r.
\label{eq:tu_kk11}
\end{equation}
(Recall that $K\theta$ is understood as in the sense of \prettyref{eq:Kid}.)
 Since $K$ is a conditional expectation operator, we have $\|h_u\|_\infty \le \|r_u\|_\infty \le ma$.
Thus together with the condition~\eqref{eq:rg_c1g}, which implies $1+\delta h_u \ge 1/2$, 
we have the following estimate on the last term in~\eqref{eq:tu_kk11}:
	$$ \left\|\delta^2 {h_u\over 1+\delta h_u}  K_1 r_u\right\|_2 \le 2\delta^2 ma \|K_1 r_u\|_2 \stepa{\le} 2\delta^2
	m^{3/2} a \tau_1 \stepb{\le} {1\over 8} \delta \sqrt{m} \tau_1\,,$$
where (a) follows from the fourth assumption of the lemma, and in (b) we also used~\eqref{eq:rg_c1g}.
Then from~\eqref{eq:tu_kk11} and the triangle inequality we have
$$ \|T_u - T_v\|_2 \ge \delta \|K_1(r_u - r_v)\|_2 - {\delta \sqrt{m} \tau_1\over 4}\,.$$
Squaring this, applying $(a-b)^2 \ge {1\over 2} a^2 - b^2$  we obtain
$ \|T_u - T_v\|_2^2 \ge {1\over2}\delta^2 \|K_1 (r_u-r_v)\|_2^2 - {1\over 16}\delta^2 m \tau_1^2 $.
Finally, applying the assumption~\eqref{eq:rg_ag} we obtain~\eqref{eq:rg_1g} with
	$$ \epsilon = {1\over 2} \tau \delta^2, \quad \epsilon_1 = \delta^2 \left({1\over 16} m \tau_1^2+ {1\over 2}
	\tau_2 \right) $$

Next, consider the following identity
$$ {f_u\over  f_0} - {f_v \over f_0}  = {1+\delta h_u \over 1+\delta \mu_u} - {1+\delta h_v \over 1+\delta \mu_v} =
	\delta {h_u - h_v\over 1+\delta \mu_u} + (1+\delta h_v) {\delta (\mu_v - \mu_u) \over (1+\delta \mu_u)(1+\delta
	\mu_v)}\,.$$
Because of~\eqref{eq:rg_c1g} we have $1+\delta h_v,1+\delta \mu_u,1+\delta \mu_v \in [\frac{1}{2},\frac{3}{2}]$. 
Thus the $\chi^2$-divergence from $f_u$ to $f_v$ is given by
\begin{align*} 
\chi^2(f_u \| f_v) & \triangleq \int d\mu f_0 {(f_u/f_0 - f_v/f_0)^2\over f_v/f_0} \\
		      &=\delta^2 \int d\mu f_0 {1+\delta \mu_v \over 1+\delta h_v} \left( {h_u - h_v\over 1+\delta \mu_u} + 
		      	{\mu_v - \mu_u \over (1+\delta \mu_u)(1+\delta \mu_v)} \right)^2\\
			&\leq C_0 \delta^2 \|h_u - h_v\|_2^2 + \delta^2 (\mu_v - \mu_u)^2 \\ 
			&\leq C_1 \delta^2 \|h_u - h_v\|_2^2\,.
\end{align*}			
where the last step follows from, by Cauchy-Schwarz,
	$$ \mu_u -\mu_v = \int dG_0 (r_u -r_v) = \int  d\mu f_0 (h_u - h_v) \le \|h_u - h_v\|_2.$$
	
For any $u$ and $v$ such that $d_{\rm H}(u,v) = 1$, $\|h_u - h_v\|_2^2 \le \gamma$ by the second assumption. 
Recall that the $\chi^2$-divergence between $n$-fold product distributions is given by 
$\chi^2(f_u^{\otimes n} \| f_v^{\otimes n}) = (1+\chi^2(f_u \| f_v))^n-1$. We have
\begin{equation}\label{eq:rg_2g}
		\chi^2(f_u^{\otimes n} \| f_v^{\otimes n}) \leq C_2 \qquad \forall u,v: d_{\rm H}(u,v)=1\,,
\end{equation}	
\apxonly{\bf Strictly speaking, once we reduce to the form of \prettyref{eq:regret3ag} using \prettyref{eq:regret2}, the sample size is $n-1$ not $n$. But OK.}
provided that
\begin{equation}\label{eq:rg_c3g}
		n \delta^2 \gamma \leq 1 
\end{equation}
Applying Assouad's lemma (see, e.g., \cite[Theorem 2.12(iv)]{Tsybakov09}), we have
\[
\inf_{\hat u \in \{0,1\}^m} \sup_{u \in \{0,1\}^m} \Expect_{G_u} \qth{d_{\rm H}(\hat u,u)} \geq C_3 m.
\]
Selecting the largest $\delta$ subject to~\eqref{eq:rg_c1g} and~\eqref{eq:rg_c3g}, the proof is completed in view of \prettyref{eq:regret3ag} and \prettyref{eq:rg_1g}.
\end{proof}

In proving our main results in both the Gaussian and Poisson model, we follow essentially the same high-level strategy: 
To apply \prettyref{prop:ass_gen}, we need to find many functions $r_1,\ldots,r_m$ which are mapped by the operator $K$ into
orthogonal ones. Thus, it is natural to take $r_k$ to be the singular functions of $K$, i.e.~those
that diagonalize the self-adjoint operator $K^*K$. Because of the general relation between $K_1$ and $K$
in~\eqref{eq:k1_def}, it will turn out that $r_k$'s indexed by \emph{sufficiently separated} $k$'s are also
mapped to orthogonal functions by the operator $K_1$, thus yielding a large subcollection of functions for \prettyref{prop:ass_gen} with $\tau=\tau_1=1$ and $\tau_2=0$. 
For the Gaussian model, we exactly follow this strategy by identifying $K^*K$ as a Mehler kernel and working with its eigenfunctions (Hermite functions). For the Poisson model, it may be difficult to compute the eigenfunctions of the resulting kernel; instead, we directly construct functions that satisfy the desired orthogonality.
We proceed to details.

	%

\section{Proofs of the main results}
\label{sec:pf}

\subsection{Normal mean model}
\label{sec:pf-gau}

In this section we prove \prettyref{th:gaussian}.
Choose the prior $G_0 = \mathcal{N}(0,s)$ for $s>0$ to be specified later. Let $\varphi(x) = {1\over
\sqrt{2\pi}}e^{-x^2/2}$ be the standard normal density. Then 
$P_{G_0}=\calN(0,1+s)$ with density $f_0(y) = {1\over \sqrt{1+s}}
\varphi\left({y\over\sqrt{1+s}}\right)$. Furthermore, the operators $K$ and $K_1$ in~\eqref{eq:k_def}
and~\eqref{eq:k1_def} are given explicitly as follows:
\begin{prop}
\label{prop:KK1-gau}	
\begin{align} 
   Kr(y) &= (r(\eta \cdot) \ast \varphi)(\eta y), \qquad \eta \triangleq \sqrt{s\over s+1}\label{eq:kdef_gsn}\\
   K_1 r(y) &= \eta^2 K(r')(y)\,,\label{eq:k1def_gsn}
\end{align}
where in~\eqref{eq:kdef_gsn} $\ast$ denotes convolution and \eqref{eq:k1def_gsn} holds for all $r$ with bounded derivative.
\end{prop}
\begin{proof}
To show~\eqref{eq:kdef_gsn}, by definition we have
$$ Kr(y) = {1\over \eta}  \int_{\reals} dx  {\varphi(y-x) \varphi({x\over \sqrt{s}})\over \varphi({y\over \sqrt{1+s}})}
r(x)\,.$$
Changing the integration variable to $x_1 = x/\eta$ and noticing that ${\varphi(y-x) \varphi({x\over
\sqrt{s}})\over \varphi({y\over \sqrt{1+s}})} = \varphi(x_1-\eta y)$ completes the proof
of~\eqref{eq:kdef_gsn}. We note that trivially then 
\begin{equation}\label{eq:n0}
	(Kx)(y) = \eta^2 y\,,
\end{equation}
which also follows from the expression of the condition mean with a Gaussian prior.
To prove~\eqref{eq:k1def_gsn}, we first notice that it is sufficient to
show that for an arbitrary (possibly non-differentiable) bounded $r$ we have
\begin{equation}\label{eq:n1}
	K_1 r(y) = {d\over dy} Kr(y)\,.
\end{equation}
Indeed, from here for differentiable $r$ we can simply apply derivative to the first term in the convolution
in~\eqref{eq:kdef_gsn} and get~\eqref{eq:k1def_gsn}. To show~\eqref{eq:n1} we take derivative
on the $\varphi$ term of the convolution and using $\varphi'(z) = -z \varphi(z)$ to get
$$ {d\over dy} Kr(y) = \eta (r(\eta \cdot) \ast \varphi')(\eta y) = \eta \int dx r(\eta x) (x-\eta
y)\varphi(x-\eta y) = (K(xr) - Kx Kr)(y)\,,$$
where we used~\eqref{eq:n0} for the second term. This proves~\eqref{eq:n1} by the definition of
$K_1$ operator.
\end{proof}

In view of~\eqref{eq:k1def_gsn} we see that to apply \prettyref{prop:ass_gen} we need to be able
to compute inner products of the form $(Kr_1, K r_2)_{L_2(f_0)}$. To simplify these computations, we
first extend (via~\eqref{eq:kdef_gsn}) the domain of $K$ to $L_2(\mreals, \Leb)$
 and then
introduce the self-adjoint operator $S=K^* K: L_2(\mreals, \Leb)\to L_2(\mreals, \Leb)$, so that
for any $f,g \in L_2(\Leb)$ we have
\begin{equation}
(K f, K g)_{L_2(f_0)} = (Sf, g)_{L_2(\Leb)}\,.
\label{eq:KtoS}
\end{equation} 
An explicit computation shows that $S$ is in fact a Mehler kernel operator, i.e.
	\begin{equation}
	Sf(x) = \int f(\tilde x) S(x,\tilde x) d\tilde x,
	\label{eq:Sop}
	\end{equation}	
with (in view of~\eqref{eq:kdef_gsn}):
\begin{equation}\label{eq:sker_def}
	S(x,\tilde x) = {1\over \eta^2} \int_{-\infty}^\infty \varphi(x/\eta - \eta y) \varphi(\tilde x/\eta
- \eta y) f_0(y)\,  dy = \lambda_1 e^{-\frac{\lambda_2}{2} (x^2 + \tilde x^2 - 2 \rho x\tilde x)}\,,
\end{equation}
where 
\begin{equation}\label{eq:sker_def2}
			\lambda_1 = {1\over 2\pi s} {1+s\over \sqrt{1+2s}}, \qquad \lambda_2 = {(1+s)^2\over s(1+2s)}\,, \qquad \rho={s\over
		1+s}\,. 
\end{equation}		
It is well-known that the Mehler kernel is diagonalized by the Hermite
functions~\cite{mehler1866ueber}. We collect all
the relevant technical details in the following lemma proved in Appendix~\ref{apx:normal}.
\begin{lemma}\label{lem:normal_main} 
For every $s>0$, there is an orthonormal basis $\{\psi_k:
k=0,1,\ldots\}$ in $L_2(\Leb)$ consisting of eigenfunctions of the operator $S$, satisfying
the following:
\begin{align} (K \psi_k, K \psi_n)_{L_2(f_0)} &= \lambda_0 \mu^k 1\{k=n\}\label{eq:n5}\\
   (K_1 \psi_k, K_1 \psi_n)_{L_2(f_0)} &= \begin{cases} 0, &|k-n| \ge 3\,,\\
					\lambda_3 (k \mu^{k-1} + (k+1) \mu^{k+1}),& k=n\,,
   				\end{cases}\label{eq:n6}\\
   \|\psi_k\|_\infty &\le \sqrt{\alpha_1}\,, \label{eq:n7}
\end{align}
where $\mu,\lambda_0,\lambda_3,\alpha_1$ are positive constants depending only on $s$, which, for 
$0<s<1/2$, satisfy\apxonly{$0<s<s_0$ for any abs constant $s_0$ works.}
\begin{equation}
\mu \asymp s, \, \lambda_0 \asymp {1\over \sqrt{s}}\,,
	\lambda_3 \asymp \sqrt{s}, \alpha_1 \asymp {1\over \sqrt{s}}\,.
\label{eq:s-dependency}
\end{equation}
In addition, for all $s>0$ we have $0<\mu < 1$.
\end{lemma}

\begin{proof}[Proof of  \prettyref{th:gaussian}]
In view of \prettyref{lmm:avgregret}, it is equivalent to consider the individual regret $\Regret_n$.
We first consider the $s$-subgaussian case. 
Without loss of generality, assume that $s<1/2$.
 By applying \prettyref{prop:ass_gen}, we aim to show the following lower bound
\[
	\Regret_n(\SubG(2s)) \geq {c\over n} {\log n \over \log {1\over s}}
\]
for some universal constant $c$.
The case of compact support then follows from taking $s \asymp \frac{1}{\log n}$ and truncation (\prettyref{lmm:truncation}).

Let $\calG_s'$ be
the set of distributions $\calG$ in \prettyref{prop:ass_gen}, namely,
$\calG_s' = \{ G: 1/2\leq \frac{dG}{dG_0} \leq 3/2\}$, where $G_0 = \calN(0,s)$. It is straightforward to verify that 
$\calG_s' \subset \SubG(2s)$ so we focus on lower bounding $\Regret_n(\calG_s')$.
We select the
required $m$ perturbations as follows:
$$ r_j = \xi_{m+3j} \psi_{m+3j}\,,\qquad j=1,\ldots,m\,,$$
where
		$$ \xi_i \triangleq {1\over \sqrt{(K_1 \psi_i, K_1 \psi_i)}}\,.$$
In view of \prettyref{lem:normal_main}, applying \prettyref{prop:ass_gen} with   $\tau = \tau_1 = 1$ and $\tau_2
= 0$ yields
\begin{equation}\label{eq:pc_7}
		\Regret_n(\calG_s') \gtrsim {m \over \max(n\gamma, m^2 a^2)}\,,
\end{equation}
	where we can take
	$$ a = \max_{m \le i \le 4m} \xi_i \|\psi_i\|_{\infty}, \,\gamma = \max_{m\le i \le
	4m} \xi_i^2 \|K\psi_i\|^2_{L_2(f_0)}\,. $$
By \prettyref{lem:normal_main}, we have
$$ \xi_i^2 = {1\over \lambda_3 (i \mu^{i-1} + (i+1)\mu^{i+1})} \asymp {1\over \sqrt{s} m \mu^{i-1}}\,.$$
Together with $\|K \psi_i\|^2 \asymp {1\over \sqrt{s}} \mu^i$ we get
\begin{equation}\label{eq:pc_8}
	\gamma \asymp {1\over m}.
\end{equation}
Similarly, we have
	\begin{equation}\label{eq:pc_9}
		a^2 = \max_{m \le i \le 4m} (\xi_i)^2 \|\psi_i\|_{\infty}^2 \lesssim {1\over s m
		\mu^{4m-1}}
\end{equation}
Finally, we select $m$ so that $n\gamma > m^2 a^2$. In view of the previous estimates, this
requirement is equivalent to ${n\over m} \gtrsim {m\over s} \mu ^{1-4m} \asymp m \mu^{-4m}$. It is
clear that taking $m = c {\log n\over \log {1\over \mu}} \asymp {\log n \over \log {1\over s}}$ for sufficiently small constant $c$ fulfills 
this requirement. In all, from~\eqref{eq:pc_7} we obtain the desired bound
\begin{equation}\label{eq:pc_10}
	\Regret_n(\calG_s') \gtrsim {1\over n} {\log n \over \log {1\over s}}\,,
\end{equation}
where the hidden constant is uniform in $0<s<1/2$.


Finally, we prove the lower bound~\eqref{eq:gsn_lb2} for priors supported on $[-h,h]$. We choose $s={c\over \log n}$ in~\eqref{eq:pc_10}
with $c>0$ to be specified shortly. Recall that every $G \in \calG'_s$ satisfies ${dG
\over dG_0} \le 3/2$, where $G_0 = \mathcal{N}(0,s)$. Thus, for every $G$ we have
$ \EE_{\theta \sim G}[\theta^4] \lesssim s^4$, while $ G[|\theta| > h] \le 3 e^{-{h^2\over 2s}}$.
Applying Lemma~\ref{lmm:truncation} we obtain then
$$ \Regret_n(\calP([-h,h])) \ge \Regret_n(\calG_s') - O(1) \sqrt{ne^{-{h^2\over 2s}}s^4}\,.$$
By selecting $c$ sufficiently small depending on $h$, we can ensure the second term is at most $O(n^{-2})$ so that \eqref{eq:pc_10} implies the required lower bound of $\Omega_h(\frac{1}{n}(\frac{\log n}{\log\log n})^2)$.
\end{proof}

\begin{remark}[Optimality of the construction]
\label{rmk:jfa} Consider $s=1$. Can we get a lower bound
better than $\Omega({\log^2 n \over n})$ by selecting other functions $r_i$ (and still with the same base prior $G_0$) in
\prettyref{prop:ass_gen}? We will argue that the answer is negative, at least if the latter
proposition is applied with $\tau=\tau_1=1$ and $\tau_2=0$ as we did above. Indeed, suppose we
found $m,a,\gamma$ such that there are functions $r_1,\ldots,r_m$ such that $\|r_i\|_\infty \le a$, $(K_1 r_i, K_1
r_j)_{L_2(f_0)} = 1\{i \neq j\}$, and $\|Kr_i\|_{L_2(f_0)}^2 \le \gamma$. Then \prettyref{prop:ass_gen} yields the regret lower bound 
\prettyref{eq:ap1}, namely, $\Regret_n \gtrsim \min\{\frac{m}{n\gamma}, \frac{1}{ma^2}\}$.
We will show in \prettyref{app:jfa} that for any such collection of functions it must be that $m \lesssim \log a$
and $\gamma \gtrsim{1\over \log a}$. This implies that the resulting lower bound is at most on the order of 
$\min\{\frac{\log^2 a }{n}, \frac{1}{ma^2}\} \lesssim \min\{\frac{\log^a }{n}, \frac{1}{ma^2}\}  \lesssim {\log^2
n\over n}$, and therefore, to be interesting we need to have $a^2 \lesssim
{n\over \log^2 n}$, or $\log a \lesssim \log n$. 
\end{remark}

\begin{remark}[Metric properties of Gaussian convolution] We can rephrase the content of Lemma~\ref{lem:normal_main} and the previous remark
also as an interesting analytical fact. Consider a convolution operator $T: L_\infty \to L_2(\mathcal{N}(0,1))$ defined as
$$ Tf = f*\mathcal{N}(0,\lambda)\,.$$
We want to understand the largest number of orthonormal functions that one can find in the image under $T$ of
an $L_\infty$-ball $B_\infty(0, a) = \{f: \|f\|_\infty \le a\}$ as $a\to \infty$. 
\apxonly{
For example, with $\lambda =
1$ we have 
$$ (Tf_1, Tf_2)_{L_2(\mathcal{N}(0,1))} = (Sf_1, f_2)_{L_2(\Leb)},, $$
where for some absolute constant $c>0$
$$ Sf(x) = c \int_{\mreals} e^{-{1\over 3}(x^2 + y^2 -xy)} g(y) dy\,.$$
}
As an attempt, one may try working with sinusoids, which are eigenfunctions of the convolution; 
however, selecting appropriately normalized sinusoids only yields $\Theta(\sqrt{\log a})$ such functions. At the same time,
Lemma~\ref{lem:normal_main} gives $\Theta(\log a)$ functions, which, as shown in \prettyref{rmk:jfa}, is in fact optimal. 
\end{remark}

\subsection{Poisson model}
\label{sec:pf-poi}
	
Consider $Y\sim \Poi(\theta)$, where $\theta$ takes values in $\reals_+$ and $Y$ takes values in $\integers_+$. 
As before, let $G_0$ denote a prior on $\theta$ and $f_0$ the corresponding probability mass function (PMF) for $Y$. In order to
have tractable expressions for $K$ and $K_1$ it seems natural to select $G_0$ to be a conjugate
prior, so we record some observations:
	\begin{itemize}
		\item $G_0$ is Gamma$(\alpha,\beta)$, with density 
			\begin{equation}\label{eq:pi0_gamma}
				G_0(x) = \frac{\beta^\alpha}{\Gamma(\alpha)} x^{\alpha-1} e^{-\beta
			x} \qquad x \in \mreals_+\,.
\end{equation}			
		\item $f_0$ is negative binomial, with PMF 
			\begin{equation}\label{eq:f0_negbin}
				f_0(y) = \binom{y+\alpha-1}{y}\left(\frac{\beta}{1+\beta}\right)^{\alpha}
			\left(\frac{1}{1+\beta}\right)^y \qquad y \in \mathbb{Z}_+
\end{equation}			
		\item The posterior distribution of $\theta$ given $Y=y$ is Gamma$(y+\alpha,\beta+1)$ 
		and thus the operator~\eqref{eq:k_def} can be written as
				$$ Kr(y) = \int_0^\infty r(x) K(x,y) dx\,,$$
			with the kernel function explicitly being given as 
				\begin{equation}\label{eq:k_poisson}
					K(x,y)=\frac{(1+\beta)^{y+\alpha+1}}{\Gamma(y+\alpha)}
				x^{y+\alpha-1} e^{-(1+\beta) x}\,. 
\end{equation}				
	\end{itemize}
In the special case of $\alpha=1$, $G_0$ is an exponential distribution with density $g_0(x) =\beta e^{-\beta x}$ and the induced Poisson mixture is geometric distribution with PMF $f_0(y) = \frac{\beta}{1+\beta} (\frac{1}{1+\beta})^y$.

\apxonly{
	Although exponential prior does not yield a tight result, we mention it as a special
	case of the above with $\alpha=1$:
	\begin{itemize}
		\item $G_0$ is Exp$(\beta)$, with density $G_0(x) = \beta e^{-\beta x}$
		\item $f_0$ is Geometric, with PMF $f_0(y) = \frac{\beta}{1+\beta} \left(\frac{1}{1+\beta}\right)^y$
		\item $P_{X|Y}(x|y)$ is Gamma$(y+1,\beta+1)$, with density $\frac{(1+\beta)^{y+1}}{y!} x^{y} e^{-(1+\beta) x}$.
	\end{itemize}
}

As before, we first derive a useful differential expression for the $K_1$ operator.
\begin{proposition}[From $K_1$ to $K$]
	\label{prop:KK1x}
	Let $G_0 = \GammaD(\alpha, \beta)$, $\alpha > 0$, then for all $r=r(x)$, smooth and bounded on
	$\mreals_+$, we have
		\begin{equation}\label{eq:pkx}
			K_1 r = {1\over 1+\beta} K(xr')\,. 
		\end{equation}			
	Conversely, suppose that for some $G_0$, which is smooth on $[0,\infty)$ and strictly
	positive on $(0,\infty)$, and all smooth bounded $r$ on $\mreals_+$ we have
		\begin{equation}\label{eq:pkx_0}
			K_1 r = c K(xr') 
		\end{equation}			
	for some constant $c$. Then $0<c<1$ and $G_0 = \GammaD(\alpha, {1\over c}-1)$ for some $\alpha>0$. 
\end{proposition}
\begin{proof}
	Fix an arbitrary smooth prior density $G_0$ strictly positive on $\mreals_+$, and denote by $f_0(y)$ the induced mixture PMF. The kernel of operator $K$ is given by 
		$$ K(x,y) = {G_0(x)\over f_0(y)} e^{-x} {x^y\over y!}\,.$$
	We have then the following easily verified identities:
	\begin{align} \label{eq:Kidgx}
		K(xg)(y) &= {f_0(y+1)\over f_0(y)} (y+1) Kg(y+1) \\
		\partial_x K(x,y+1) &= K(x,y) {f_0(y)\over f_0(y+1)} w(x,y) \qquad w(x,y) \triangleq
		1+(\partial_x \ln G_0(x) - 1){x\over y+1} 	\label{eq:Kidg2x}\\
		f_0(y+1) K(x,y+1) &= f_0(y) K(x,y) {x\over y+1} \label{eq:Kidgx3}
	\end{align}
	From~\eqref{eq:Kidgx} we have
		$$ K(xr')(y) = {f_0(y+1)\over f_0(y)} (y+1) K(r')(y+1)\,.$$
	On the other hand, 
	since $r$ is bounded, integrating by parts and noting that $K(0,y+1) = 0$ for all $y\ge 0$ we obtain\footnote{$K(\infty,y+1)=0$ may not hold. Is integration by parts OK? Should we also require $G_0$ to be bounded to be safe?}
	$$ K(r')(y+1) = \int_{\mreals_+} dx r'(x) K(x,y+1) = -\int_{\mreals_+} dx\,  r(x) \partial_x K(x,y+1)\,.$$
	Invoking~\eqref{eq:Kidg2x} we overall get
	\begin{equation}\label{eq:pkx_1}
		K(xr')(y) = -(y+1) \int dx r(x) K(x,y) w(x,y) \, dx\,.
	\end{equation}		
	
	Next, we apply~\eqref{eq:Kidgx} to get
	$$ Kx(y) = {(y+1)f_0(y+1)\over f_0(y)}\,.$$
	Then, from~\eqref{eq:k1_def} we obtain that
	$$ K_1 r (y) = K(xr)(y) - {(y+1)f_0(y+1)\over f_0(y)} Kr(y)\,.$$
	Applying~\eqref{eq:Kidgx} to the first term we get
	\begin{equation}\label{eq:pkx_2}
		K_1 r(y) = {(y+1)f_0(y+1)\over f_0(y)} (Kr(y+1) - Kr(y))\,.
	\end{equation}		

	Via~\eqref{eq:pkx_1} and~\eqref{eq:pkx_2}, we observe that~\eqref{eq:pkx_0} is equivalent
	to 
	$$ {(y+1)f_0(y+1)\over f_0(y)} \int_{\mreals_+} dx r(x) (K(x,y+1) - K(x,y)) = -c \int dx r(x) (y+1)
	K(x,y) w(x,y)\,,$$
	which needs to hold for all $r(x)$ and $y\ge 0$. Thus,~\eqref{eq:pkx_0} is equivalent to 
	$$ {f_0(y+1) \over f_0(y)} (K(x,y+1)-K(x,y)) = -c w(x,y) K(x,y)\,.$$
	Apply~\eqref{eq:Kidgx3} to replace $K(x,y+1)$ with $K(x,y)$ to get:
	$$ \left({x\over y+1}-{f_0(y+1)\over f_0(y)}\right) K(x,y) = -c w(x,y) K(x,y)\,.$$
	Since $G_0(x)>0$ for all $x>0$, we also get that $K(x,y)>0$ and, thus,
	overall~\eqref{eq:pkx_0} is equivalent to 
	\begin{equation}\label{eq:pkx_3}
		x(1+c(\partial_x \ln G_0(x) - 1)) = (y+1)\left(-c+ {f_0(y+1)\over f_0(y)}\right)\,.
	\end{equation}	
	From here we first notice that when $G_0 = \GammaD(\alpha,\beta)$ and $f_0$ is the
	negative binomial both sides evaluate to $\alpha-1\over 1+\beta$ if $c={1\over 1+\beta}$.
	This proves~\eqref{eq:pkx} and the first part of the proposition. For the second part, notice that in order
	for~\eqref{eq:pkx_3} to hold for all $x,y$ both sides must be constants, implying that
	that $f_0(y)$ is a negative binomial PMF and, thus, the  $G_0$ is a Gamma distribution.
\end{proof}

The converse part of \prettyref{prop:KK1x} justifies yet again our focus on the Gamma prior $G_0$. 
In Appendix~\ref{apx:poisson} we proceed by (almost) diagonalizing the self-adjoint compact
integral operator $K^*K$. The result of this analysis is summarized in the following.

\begin{lemma}\label{lem:poisson_main} Fix $\delta > 0$. 
Let  $G_0 =\GammaD(\alpha,\beta)$. 
Then there exist absolute positive constants
$C,m_0$ such that for all $m\ge m_0$, $\beta\geq 2$ and $\alpha \geq 4m$, there are functions $r_1,\ldots,r_m$ with the following properties:
\begin{align} \|K_1 r_j\|_{L_2(f_0)}^2 &= 1\,, \label{eq:lpm_1}\\
   (K r_j, Kr_i)_{L_2(f_0)} &= (K_1 r_j, K_1 r_i)_{L_2(f_0)} = 0 \quad \forall i\neq j\,\label{eq:lpm_2}\\
	\|K r_j\|_{L_2(f_0)}^2  &\le {C \beta \over \alpha m} \label{eq:lpm_3}\,,\\
   \|r_j\|_\infty & \le \sqrt{\frac{\beta}{\alpha}} e^{C (m \log \beta + \alpha)}\,. \label{eq:lpm_4}
\end{align}
\end{lemma}

\begin{lemma}\label{lem:poisson_exp} In the special case of $\alpha=1$ (so that $G_0$ is
the exponential distribution with mean $1/\beta$) for any fixed $\beta>0$ there exists a constant $C=C(\beta)>0$ such that for
all $m\ge 1$ there exist functions $r_1,\ldots, r_m$
satisfying~\eqref{eq:lpm_1},~\eqref{eq:lpm_2} and
\begin{align} 
	\|K r_j\|_{L_2(f_0)}^2  & \le {C\over m^2} \label{eq:lpm_3e}\,,\\
   \|r_j\|_\infty & \le e^{Cm}\,. \label{eq:lpm_4e}
\end{align}
\end{lemma}

Proofs of Lemmas \ref{lem:poisson_main} and \ref{lem:poisson_exp} are found in Appendix~\ref{apx:poisson}. 
Assuming these lemmas we are in position to prove the lower bound in Theorem~\ref{th:poisson}.
(The upper bound is proved by analyzing the Robbins estimator in \prettyref{app:robbins}.)
At the high level, for the lower bound in the compactly supported case we will choose the base prior $G_0=\GammaD(\alpha,\beta)$ with $\alpha =\Theta(\log n)$ and $\beta = \Theta(\log n)$. In this case, $G_0$ is concentrated on a constant mean with variance $\Theta(\frac{1}{\log n})$; this is conceptually similar to the construction for the Gaussian model with bounded means where we choose $G_0=N(0,\Theta(\frac{1}{\log n}))$.
For the subexponential class, however, we choose an exponential $G_0$ (with $\alpha=1$ and $\beta=\Theta(1)$), which is analogous to $G_0=N(0,\Theta(1))$ in the Gaussian model with subgaussian priors.

\begin{proof}[Proof of Theorem~\ref{th:poisson}: Lower bound]
We start with the compactly supported case. 
Let $\calG=\calP([0,h])$ denote the collection of all priors on $[0,h]$.
First, in view of Lemma~\ref{lmm:avgregret} the lower bound is equivalent to proving that 
	$$ \Regret_n(\calG) \gtrsim {1\over n} \left({\log n \over \log
		\log n}\right)^2\,. $$

We aim to apply Proposition~\ref{prop:ass_gen} with functions $r_1,\ldots,r_m$ chosen from
Lemma~\ref{lem:poisson_main} applied with 
$$  m = c_1 {\log n \over \log \log n}, \quad \alpha = c_1 \log n\,, \quad \beta = c_2 \alpha, $$
with constants $c_1,c_2,c_3>0$ to be specified later based on $h$. Properties~\eqref{eq:lpm_1}
and~\eqref{eq:lpm_2} ensure that $\tau_2 = 0$, $\tau=\tau_1=1$ and $\gamma = {C c_2\over m}$ for some absolute constant $C$ (in the notation of the
Proposition~\ref{prop:ass_gen}). Setting $\calG' = \{G: |{dG\over dG_0} - 1| \le 1/2\}$,
Proposition~\ref{prop:ass_gen} implies that
\begin{equation}\label{eq:gc_6}
		\Regret_n(\calG') \gtrsim {m^2\over n} \asymp {1\over n} \left({\log n \over \log
		\log n}\right)^2\,,
\end{equation}	
provided that 
	\begin{equation}\label{eq:gc_5}
		m \sqrt{c_2} e^{C (m \log \beta + \alpha)} < \sqrt{n\gamma} = \sqrt{Cc_2 n/m}\,.
	\end{equation}
	Note that as $n\to \infty$, the left side is $n^{2Cc_1+o(1)}$ and the right side is $n^{1/2+o(1)}$. 	
Thus, taking $c_1 = {1\over 8C}$ ensures that, for all sufficiently large $n\ge n_0$,
condition~\eqref{eq:gc_5} and, in turn, the bound~\eqref{eq:gc_6} hold.

To translate~\eqref{eq:gc_6} to regret on the set of priors $\calG=\{G: \supp(G) \subset [0,h]\}$ we 
apply Lemma~\ref{lmm:truncation} (with $a=h$). Note that for any $G \in \calG'$ we have 
$G[\theta > h] \le 2 G_0[\theta > h]$.
For $G_0=\text{Gamma}(\alpha,\beta)$ with mean $\frac{\alpha}{\beta}=\frac{1}{c_2}$, Chernoff bound yields
\begin{equation}
G_0[\theta > h] \le \inf_{t>0} e^{-th} \pth{1-\frac{t}{\beta}}^{-\alpha} = e^{-h\beta + \alpha + \alpha \log \left({h\beta\over \alpha}\right)} = e^{-\alpha(hc_2 + 1-\log(hc_2))},
\label{eq:chernoff-gamma}
\end{equation}
provided that $h > \frac{1}{c_2}$. 
Recall that $\alpha=c_1\log n$.
Thus, choosing $c_2=\kappa /h$ with $\kappa=\kappa(c_1)>1$ such that  $\kappa+ 1-\log(\kappa) = {4\over c_1}$, we get
	$ G[\theta > h] \le 2n^{-4}$.
On the other hand, we have
	$$ \EE_{G_0}[\theta^4] = {1\over \beta^4} \int_0^\infty {y^{\alpha + 3}\over
	\Gamma(\alpha)} e^{-y} dy = {\Gamma(\alpha+4)\over \Gamma(\alpha) \beta^4} \asymp c_2^{-4}\,.$$
Overall, from~\eqref{eq:gc_6} and Lemma~\ref{lmm:truncation} we obtain that  
	$$ \Regret_n(\calG) \gtrsim {1 \over n} \left({\log n \over \log
		\log n}\right)^2 - O(n^{-3/2})\,, $$
	completing the proof of the lower bound~\eqref{eq:poi_lb}.

Next we address the subexponential case. Recall that $\SubE(s)$ denotes the collection of all priors $G$ such that $G(X\geq t) \leq 2e^{-t/s}$.
Choose the base prior $G_0$ to be $\Gamma(1,\beta)$, namely, the exponential distribution with parameter $\beta$.
Then $\calG' = \{G: |{dG\over dG_0} - 1| \le 1/2\}\subset\SubE(s)$.
The rest of the argument is the same by applying \prettyref{prop:ass_gen} with \prettyref{lem:poisson_exp} in place of \prettyref{lem:poisson_main}.
Specifically, let $\alpha=1,\beta=s,m=c \log n$.
Thanks to \prettyref{lem:poisson_exp}, we may apply \prettyref{prop:ass_gen}
with $\tau_2 = 0$, $\tau=\tau_1=1$ and $\gamma = {C \over m^2}$, where $C=C(\beta)$.
Choose $c$ such that $m e^{Cm} \leq \sqrt{n/m}$, we obtain
$\Regret(\SubE(s))\geq \Regret(\calG')\gtrsim \frac{m^3}{n} \asymp \frac{\log^3 n}{n}$, with hidden constants depending on $s$ only.
\end{proof}

\section{Discussion}
	\label{sec:discuss}
	
In this paper we proved regret lower bounds for nonparametric empirical Bayes estimation.
For the Poisson model (\prettyref{th:poisson}), this is achieved by the Robbins estimator
\prettyref{eq:robbins} which certifies its optimality. For the Gaussian model
(\prettyref{th:gaussian}), the rate for the optimal regret is open, but given Remark~\ref{rmk:jfa}
it is likely that the improvement of the upper bounds is the bottleneck. 

Our program for proving regret lower bound (provided by \prettyref{prop:ass_gen}) is general and applies to
arbitrary models. However, its
execution depends on the choice of a good base prior, such that the resulting
kernels $K$ and $K_1$, cf.\prettyref{eq:k_def} and
\prettyref{eq:k1_def}, are easy to handle. Namely, the lower bound requires understanding
the decay of the eigenvalues and the growth of the eigenfunctions of these kernels.  As
such, this is the main difficulty of extending our work to more
general models, e.g. exponential families considered in
\cite{Singh79,pensky1999nonparametric,li2005convergence}.
	
To end this paper, we discuss related problems in empirical Bayes with the hope of offering different perspectives and identifying open questions.

\paragraph{Restricted EB.} 
In this paper we adopt the general EB formulation where the goal is to approach the Bayes risk \prettyref{eq:mmse}. In comparison, restricted EB aims to approach the minimum risk in a given class of estimators, such as linear or thresholding estimators. This perspective provides EB interpretations for a number of commonly used methodologies for denoising or model selection, including James-Stein estimator, FDR, AIC/BIC, etc; see \cite{efron1972empirical,george2000calibration,johnstone2004needles,jiang2009general}. 
It is of interest to understand the regret in restricted EB settings. Notably, the reduction to individual regret in \prettyref{sec:reduction} no longer holds.

%
%

\paragraph{Multiple testing.} 
 In addition to estimation, multiple testing and variable selection \cite{george2000calibration,yuan2005efficient} are important formulations in empirical Bayes.
As a specific problem akin to ours, consider the normal means problem where for each $\theta_i$ we aim to test the hypothesis $H_{0,i}: \theta_i \leq \theta_0$ and $H_{1,i}: \theta_i > \theta_0$. For any test procedure $\delta\in\{0,1\}$, a meaningful loss function is $R(\theta,\delta) \triangleq \delta (\theta_0-\theta) \indc{\theta \leq \theta_0} + (1-\delta) (\theta-\theta_0) \indc{\theta > \theta_0}$. An $O((\log n)^{3/2})$ upper bound on the total regret is shown in  \cite{liang2000empirical} (see also \cite{li2002empirical}).
It will be interesting to obtain regret lower bound for such testing problems. Note that for this non-quadratic loss, the reduction \prettyref{eq:regret2}	to estimating regression function does not hold.

\paragraph{Compound regret.} 
As shown in \prettyref{prop:regret-comp} and \prettyref{prop:comp-subg}, the EB regret lower bounds in this paper can be extended to the compound setting. In comparison, regret upper bounds are less understood. For the Poisson model, it is not obvious whether the analysis of Robbins estimator in \prettyref{app:robbins} applies to the compound regret in \prettyref{eq:totregretcomp1}.
For the Gaussian model, converting the multiplicative bound of \cite{jiang2009general} into additive form yields a regret that scales at least as $\sqrt{n}$. As such, achieving a logarithmic compound regret is currently open.


	%
%
%
%
	
	\paragraph{Sequential version.}
It is of practical interest to consider the sequential version of the EB or compound estimation problem (see, e.g., \cite{hannan1957,samuel1965sequential,van1966sequential,gilliland1968sequential}), in which case the estimation of each parameter $\theta_t$ can only depend on the data received so far.
	To this end, consider the \emph{accumulative regret}, which is a variant of the total regret \prettyref{eq:avgregret} 
		\begin{equation}
	\AccRegret_n(\calG) \triangleq \inf_{\hat \theta^n} \sup_{G \in \calG} \sth{\Expect_G\qth{\|\hat \theta^n(Y_1,\ldots,Y_n) - 
\theta^n\|^2} - n \cdot \mmse(G)}\,.
	\label{eq:accregret}
	\end{equation}
	with the causality constraint that $\hat \theta_t = \hat \theta_t(Y_1,\ldots,Y_t)$ for each $t=1,\ldots,n$.
	The following simple result relates the \emph{accumulative regret} to individual regret. 
	For example, for bounded normal means problem, 
	\prettyref{th:gaussian} shows that the 	accumulative regret is between $(\frac{\log n}{\log\log n})^2$
	and $\frac{(\log n)^3}{(\log\log n)^2}$.
	
\begin{lemma}
\label{lmm:accregret}	
	\[
	n \cdot \Regret_n(\calG) \leq \AccRegret_n(\calG)  \leq  \sum_{k=1}^n \Regret_k(\calG).
	\]
\end{lemma}
The lower bound follows by dropping the causality constraint and using \prettyref{lmm:avgregret}. The upper bound simply follows from applying the optimal estimator that achieves the individual regret $\Regret_k$ for sample size $k$. It is unclear whether this strategy is optimal, because the worst-case ``hyperprior'' (prior on the set of priors) for each individual regret depends on the sample size, while in \prettyref{eq:accregret} the prior is frozen throughout all sample sizes.


\paragraph{$\ell_p$-ball and sparsity.}
For the normal mean model \prettyref{th:gaussian} provides a regret lower bound for subgaussian priors.
It is of great interest to consider the moment constraint due to its connection to sparsity and adaptive estimation.
Consider the EB setting, where the prior $G$ has bounded $p$th absolute moment, namely, 
$\calM_{p,\alpha} \triangleq \{G: \int |\theta|^p G(d\theta)\leq \alpha^p\}$.
\cite[Theorem 3]{jiang2009general} showed that for this class the individual regret is at most 
$\Regret(\calM_{p,\alpha})=\tilde O((\alpha/n)^{\frac{p}{1+p}})$. (Here and below tilde hides logarithmic factors.)
Determining the tightness of this bound is an open question.
In comparison, under the same moment condition, the minimax squared Hellinger risk for density estimation is also $\tilde O((\alpha/n)^{\frac{p}{1+p}})$ 
(see \cite[Theorem 1]{zhang2009generalized} or \cite[Theorem 4]{jiang2009general}). This dependency on $n$ is optimal within logarithmic factors as shown by \cite[Theorem 2.3]{kim2020minimax} for constant $\alpha$.

The connection to sparse estimation lies in the compound setting. Suppose the mean vector belongs to the $\ell_p$-ball 
$\Theta_{n,p,\alpha} =\{\theta^n \in\reals^n: \sum_{i=1}^n |\theta_i|^p \leq n\alpha\}$, where $0<p<2$.
The minimax squared error of estimating $\theta^n$ on the basis of $Y^n \sim \calN(\theta^n,I)$ is known to be 
 $\tilde\Theta(n \alpha^p)$ (see \cite[Theorem 3]{DJ94} for precise characterizations).
\cite[Theorem 2]{jiang2009general} showed that EB estimator based on NPMLE can adaptively achieve the minimax risk over $\Theta_{n,p,\alpha}$ up to an $1+o(1)$ multiplicative factor, provided that $\alpha = \tilde\Omega(n^{-1/p})$; this condition can be explained by comparing the regret $(\alpha/n)^{\frac{p}{1+p}}$ with the minimax risk per coordinate  $\alpha^p$.
Using arguments similar to \prettyref{prop:comp-subg} one can show that the regret in the compound setting is at least that in the EB setting.
Thus proving regret lower bound sheds light on to what extent can one adapt to the radius of the $\ell_p$-ball and the optimality of NPMLE-based schemes.

\apxonly{

Open questions: Benjamini-Hochberg (FDR) fundamental limits. Here (one possible model is that) we have $X_i \simiid \Bern(p)$ and $P_{Y|X=1} =
p_1$, $P_{Y|X=0}=p_0$ where $p_0$ is fixed and known (uniform on $[0,1]$ in applications), while $p$ and $p_1$ are
unknown. Again, the goal is to minimize regret wrt Bayesian guy (who knows $p$ and $p_1$). There is an extra twist there
that in FDR one is interested only in estimators that provably satisfy $\PP[\hat X =1 | X=0] \le q$ uniformly over
$p,p_1$. However, we may at first ignore this constraint.

One possible way to formulate the problem within the confines of the regret framework is the following:
Each $\theta$ is a tripe: $\theta=(Y_0,Y_1,B) \sim G=p_0\otimes p_1 \otimes \Bern(p)$. The observation $X|\theta$ is given by $X=Y_B$ (note that this channel is deterministic). 
Then two versions of regret:
\begin{itemize}
	\item Regret without FDR control:
	\[
	\inf_{\hat X^n} \sup_G \frac{1}{n} \Expect[d_{\rm H}(\hat X^n,X^n)] - R_{\text{Bayes}}(G)
	\]
	where $R_{\text{Bayes}}(G) = \min_{\hat X} \prob{\hat X\neq X} $.
	
	\item Regret with FDR control:
	\[
	\inf_{\hat X^n} \sup_G \frac{1}{n} \Expect[d_{\rm H}(\hat X^n,X^n)] - R_{\text{Bayes}}^{\text{FDR}}(G;\alpha)
	\]
	where the decision rule $\hat X^n$ satisfies the FDR constraint that 
	\[
	\expect{\frac{\sum_{i=1}^n \indc{X_i=0,\hat X_i=1}}{1 \vee \sum_{i=1}^n \indc{X_i=0} }} \leq \alpha, \qquad \forall G
	\]
	and
	$R_{\text{Bayes}}^{\text{FDR}}(G;\alpha) = \min_{\hat X} \{\pprob{\hat X=0|X=1}: \pprob{\hat X=1|X=0} \leq \alpha  \} $.\footnote{\nb{The formulation is subject to change. I actually don't understand why we define the Bayes risk this way as opposed to $\pprob{\hat X\neq X}$ above.}}

}

\appendix

\section{Proofs for the Gaussian model}

\subsection{Proof of Lemma~\ref{lem:normal_main}}
\label{apx:normal}


Let $\{H_k\}$ denote the monic orthogonal polynomial with respect to the standard normal density
$\varphi$, i.e. $H_k(y) = (-1)^k e^{y^2/2} {d^n\over dy^n} e^{-y^2/2}$, so that where $H_0(y)=1,H_1(y)=y,H_2(y)=y^2-1,\ldots$. 

We start by recalling the following facts about Hermite polynomials that will be used in the
proof:\footnote{These facts are adapted from \cite[Chap.~7 and 8]{GR}. Therein, the Hermite
polynomials, denoted here by $\bfH_k$, are orthogonal with respect to $e^{-x^2}$. Thus $H_k(x) = 2^{-k/2}\bfH_k({x}/{\sqrt{2}})$.}
\begin{itemize}
	\item Orthogonality~\cite[7.374]{GR}:
\begin{equation}
 \int \varphi(x) H_i(x) H_j(x) = j! \indc{i=j}. 
\label{eq:Hk-Iprod}
\end{equation}


\item Three-term recursion~\cite[8.952.2]{GR}:
\begin{equation}
x H_k(x) = H_{k+1}(x) + k H_{k-1}(x).
\label{eq:Hk-three}
\end{equation}

	\item Derivatives~\cite[8.952.1]{GR}:
\begin{equation}
H_k'(y) = kH_{k-1}(y),
\label{eq:Hk-deriv}
\end{equation}

\item Cram\'er's inequality \cite[8.954.2]{GR}:
\begin{equation}
	|H_k(y)| \leq \kappa \sqrt{k!} e^{\frac{y^2}{4}},
	\label{eq:cramer}
\end{equation}
where $\kappa \approx 1.086$ is an absolute constant.

\item Mehler formula~\cite{mehler1866ueber}: For any $0<\mu<1$ and $u,v \in \mreals$ we have
\begin{equation}\label{eq:mehler}
	\sum_{k=0}^\infty {\mu^k\over k!} H_k(u) H_k(v) \varphi(u) \varphi(v) = {1\over 2\pi
\sqrt{1-\mu^2}} e^{-{1\over 2(1-\mu^2)}(u^2 + v^2 - 2\mu u v)}\,.
\end{equation}
\end{itemize}

From~\eqref{eq:mehler} by dividing both sides by $\sqrt{\varphi(u)\varphi(v)}$ we obtain
\begin{equation}\label{eq:mehler2}
	\sum_{k=0}^\infty {\mu^k\over k!} g_k(u) g_k(v)=
	{1\over \sqrt{2\pi(1-\mu^2)}} e^{-{a_1\over 2}(u^2 + v^2 - 2c_1uv)}, \qquad a_1\triangleq{1\over2}{1+\mu^2\over
	1-\mu^2}, \, c_1\triangleq{2\mu\over 1+\mu^2}\,,
\end{equation}
with $g_k(u) \triangleq H_k(u)\sqrt{\varphi(u)}$, which forms a complete set of orthogonal functions in
$L_2(\Leb)$, cf.~\eqref{eq:Hk-Iprod}. Our goal is to find constants $(\mu,\alpha_1,\alpha_2)$ such that
\begin{equation}\label{eq:n3}
	\alpha_2 \sum_k {\mu^k\over k!} g_k(\alpha_1 x) g_k(\alpha_1 \tilde x) = S(x,\tilde x)\,,
\end{equation}
with $S(x,\tilde x)$ defined in \prettyref{eq:sker_def},
in which case $\{g_k(\alpha_1 x)\}$ will be the orthogonal eigenbasis of the $S$ operator in \prettyref{eq:Sop}. By
comparing~\eqref{eq:mehler2} and the definition of $S(x,\tilde x)$, we need to fulfill
\[
c_1 = \rho, \quad a_1 \alpha_1^2 = \lambda_2
\]
where $\rho=\frac{s}{1+s}, \lambda_1 = {1\over 2\pi s} {1+s\over \sqrt{1+2s}}$ and $\lambda_2 = {(1+s)^2\over s(1+2s)}$ as in~\eqref{eq:sker_def2}.
Using \prettyref{eq:mehler2}, we find
\begin{equation}\label{eq:n2}
	\mu = {\rho\over 1+\rho_1}\,, 
	 \alpha_1 = \sqrt{2 \lambda_2\rho_1}\,, \alpha_2 = \lambda_1 \sqrt{2\pi (1-\mu^2)}\,,
	\quad\rho_1 \eqdef  \sqrt{1-\rho^2}\,,
\end{equation}
which satisfies the desired \eqref{eq:n3}; in particular, we have $0<\mu<1$.
Finally, setting
\begin{equation}\label{eq:n8}
	\psi_k(x) \triangleq {g_k(\alpha_1 x)\over \|g_k(\alpha_1 \cdot)\|_{L_2(\Leb)}} = \sqrt{\alpha_1\over
k!} H_k(\alpha_1 x) \sqrt{\varphi(\alpha_1 x)} 
\end{equation}
completes the construction of the orthonormal eigenbasis of $S$. Indeed, \eqref{eq:n3} is
then rewritten as 
\begin{equation}\label{eq:n4}
	S(x,\tilde x) =  \sum_{k=0}^\infty \lambda_0 \mu^k \psi_k(x) \psi_k(\tilde x)\,,
\end{equation}
where 
\begin{equation}
\lambda_0 \triangleq {\alpha_2\over \alpha_1} = {1\over \sqrt{2\pi s (1+\rho_1)}}\,,
\label{eq:lambda0}
\end{equation}
where we used $1-\mu^2 = {2\rho_1\over 1+\rho_1}$ to simplify the expression.
In turn,~\eqref{eq:n4} implies that 
\begin{equation}\label{eq:n5b}
	S \psi_k = \lambda_0 \mu^k \psi_k\,,
\end{equation}
i.e. that $\{\psi_k\}$ is the orthonormal eigenbasis for $S$.

We proceed to checking properties~\eqref{eq:n5}-\eqref{eq:n7} of the $\psi_k$ claimed in the statement of the lemma. In view of \prettyref{eq:KtoS}, \eqref{eq:n5} is just a restatement of~\eqref{eq:n5b}. For~\eqref{eq:n6} we notice that 
$$ (K_1 \psi_k, K_1 \psi_n)_{L_2(f_0)} \overset{\eqref{eq:k1def_gsn}}{=}
\eta^4 (K \psi_k', K \psi_n')_{L_2(f_0)} 
 \overset{\prettyref{eq:KtoS}}{=} \eta^4 (S \psi_k', \psi_n')_{L_2(\Leb)}\,.$$
On the other hand, differentiating~\eqref{eq:n8} and using~\eqref{eq:Hk-deriv}
and~\eqref{eq:Hk-three} we get
\begin{align*} \psi_k'(x) &= \sqrt{\alpha_1\over k!}\sqrt{\varphi(\alpha_1 x)}\left[ \alpha_1 H_k'(\alpha_1 k)
- H_k(\alpha_1 x) {\alpha_1^2\over 2} x\right] \\
&=\sqrt{\alpha_1\over k!}\sqrt{\varphi(\alpha_1 x)}\left[ \alpha_1 k H_{k-1}(\alpha_1 x) -
{\alpha_1\over 2}(H_{k+1}(\alpha_1 x) + k H_{k-1}(\alpha_1 x))\right]\\
&= {\alpha_1 \over 2} [\sqrt{k} \psi_{k-1}(x) - \sqrt{k+1} \psi_{k+1}]\,.
\end{align*}
This implies that in the basis $\psi_k$ the operator $K_1^* K_1$ is tridiagonal, i.e. 
	\begin{align} (K_1 \psi_k, K_1 \psi_m)_{L_2(f_0)} 
		&= \lambda_3 \cdot
		\begin{cases}
			- \sqrt{(k-1) k} \mu^{k-1}, & m=k-2\\
			k \mu^{k-1} +
			(k+1)\mu^{k+1}, & m=k\\
			-\sqrt{(k+1)(k+2)} \mu^{k+1}, & m=k+2\\
			0, & \mbox{o/w}
		\end{cases}\,,
		\label{eq:K1K1}
	\end{align}
where 
\begin{equation}
\lambda_3 = \lambda_0 {\alpha_1^2 \eta^4\over 4} = {1\over \sqrt{8\pi}} \sqrt{s\over 1+\rho_1} {\rho_1 \over 1+2s}\,.
\label{eq:lambda3}
\end{equation}
This proves~\eqref{eq:n6}.

To show~\eqref{eq:n7}, we simply apply~\eqref{eq:cramer} to~\eqref{eq:n8} 
$$ |\psi_k(x)| \le \kappa \sqrt{\alpha_1\over k!} \sqrt{k!} e^{\alpha_1 x^2\over 4}
\sqrt{\varphi(\alpha_1 x)} = {\kappa\over (2\pi)^{1/4}} \sqrt{\alpha_1} \le \sqrt{\alpha_1}\,.$$
(In other words, $L_2$-normalized Hermite functions are uniformly bounded.)

Finally, \prettyref{eq:s-dependency} follows from the expressions of $\mu,\lambda_0,\lambda_3,\alpha_1$ in \prettyref{eq:n2}, \prettyref{eq:lambda0}, and \prettyref{eq:lambda3}, by noting that $\rho=s(1+o(1))$ and 
$\rho_1\to 1$ as $s\to 0$.

\subsection{Proof for \prettyref{rmk:jfa}}
\label{app:jfa}

In this appendix we justify the claim about $m$ and $\gamma$ in \prettyref{rmk:jfa}. Namely, 
suppose there exist functions $r_1,\ldots,r_m$ such that $\|r_i\|_\infty \le a$, $(K_1 r_i, K_1
r_j)_{L_2(f_0)} = 1\{i \neq j\}$, and $\|Kr_i\|_{L_2(f_0)}^2 \le \gamma$. Then we must have $m \lesssim \log a$
and $\gamma \gtrsim{1\over \log a}$. 

To this end, 
we develop $r_i$'s in the basis $\{\psi_k\}$ 
as follows:
	$$ r_i = \sum_k \rho_{i,k} \psi_k\,,$$
	where $\{\psi_k\}$ are (dilated) Hermite functions in \prettyref{eq:n8} and satisfy \prettyref{lem:normal_main}.
	It turns out that Hermite functions have polynomial-size $L_p$-norms,
	namely~\cite[Theorem 2.1]{aptekarev2012asymptotics}:
	\begin{equation}\label{eq:pc_6}
		\|\psi_k\|_{L_p(\Leb)} \asymp \alpha_1^{-{1\over p}} k^{2-p\over 4p}\,, \qquad
	0<p<4\,,
	\end{equation}	
	where by \prettyref{eq:s-dependency} $\alpha_1\asymp 1$ since $s=1$.
	Thus, $|\rho_{i,k}| = |(r_i,\psi_k)| \le \|r_i\|_\infty \|\psi_k\|_1 \lesssim a k^{1/4}$. 
	Now consider 
	$$ (K_1 r_i, K_1 r_j)_{L_2(f_0)} = \sum_{k_1,k_2} \rho_{i,{k_1}} \rho_{j, k_2} (K_1 \psi_{k_1}, K_1 \psi_{k_2})\,.$$
	From~\eqref{eq:K1K1} and the estimate $|\rho_{i,k}| \lesssim a k^{1/4}$ we conclude that
	in the preceding sum the total contribution of the terms with $k_1,k_2 \ge k_0$ can be
	bounded as $\lesssim a^2 k_0^{O(1)} \mu^{k_0}$. Thus, selecting $k_0 \asymp \log a$ we can define truncated versions
	$$ \tilde r_i = \sum_{k \le k_0} \rho_{i,k} \psi_k\,.$$
	For these $m$ functions we have
	$$ (K_1 \tilde r_i, K_1 \tilde r_j)_{L_2(f_0)} = 1\{i\neq j\} + o(1) $$
	as $a\to \infty$. So on one hand, the set of $m$ functions $\{K_1\tilde r_i, i \in [m]\}$
	is almost orthonormal, and on the other hand it is contained in the span of 
	functions $\{K_1 \psi_j, 0\le j\le k_0\}$. This implies $m \le k_0+1 \lesssim \log a$. 
	
	Next, we show the claim for $\gamma$. A similar truncation argument (this time leveraging the
	diagonal structure of $K$ in~\eqref{eq:n5}) shows that $\|K\tilde
	r_i\|_{L_2(f_0)}^2 = (1+o(1)) \|Kr_i\|_{L_2(f_0)}^2$. Since $\tilde r_i \in
	V \eqdef \mathrm{span}\{\psi_1,\ldots,\psi_{k_0}\}$ we have
		$$ {1+o(1)\over \gamma} \le \max_{f \in V} {\|K_1 f\|^2_{L_2(f_0)}\over \|K
		f\|^2_{L_2(f_0)}} =: \xi\,.$$
	Standard manipulations then show that $\xi$ equals maximal eigenvalue of the
	operator $S^{-{1\over 2}} K_1^* K_1 S^{-{1\over 2}}$ restricted to the subspace $V$. This $k_0\times
	k_0$ matrix is (as
	can be seen from~\eqref{eq:n5} and~\eqref{eq:K1K1}) tridiagonal with $k$-th row's non-zero
	entries being $\{-\sqrt{(k-1)k}, {k\over \mu} + \mu(k+1), -\sqrt{(k+1)(k+2)}\}$. Bounding the
	maximal eigenvalue by the maximal absolute row-sum, we get $\xi \lesssim k_0 \lesssim \log
	a$, as required.



\section{Proofs for the Poisson model}\label{apx:poisson}

In this appendix we prove supporting Lemmas \ref{lem:poisson_main} and \ref{lem:poisson_exp} for the Poisson model. 
Similar to the proof of \prettyref{lem:normal_main}, we first derive the explicit expression of the operator $S=K^*K: L_2(\integers_+, f_0)\to L_2(\mreals_+, \Leb)$ 
that satisfies
\begin{equation}
(K f, K g)_{L_2(\integers_+, f_0)} = (Sf, g)_{L_2(\mreals_+, \Leb)}\,.
\label{eq:KtoS-poi}
\end{equation} 

.
	\begin{proposition}[Kernel under the Gamma prior]
	\label{prop:Skernel-poi}
		Let $G_0 = \GammaD(\alpha,\beta)$ and $f_0$ be the induced negative binomial
		distribution~\eqref{eq:f0_negbin}. Let $K$ denote the operator in~\eqref{eq:k_def} which acts
		as $K: L_2(\mreals_+, \Leb) \to L_2(\mathbb{Z}_+, f_0)$. Then the operator $S\eqdef
		K^*K: L_2(\mreals_+,\Leb) \to L_2(\mreals_+,\Leb)$ satisfies $S r(x) =
		\int_0^\infty S(x,x') r(x') dx'$ with kernel function given by
				\begin{equation}\label{eq:spoi_ddef}
					S(x,x') = C(\alpha,\beta) e^{-(x+x')(1+\beta)}
				((1+\beta)xx')^{\alpha-1\over 2} I_{\nu}(2\sqrt{(1+\beta)xx'}),
				\quad \nu\eqdef \alpha-1\,,
\end{equation}				
				where $C(\alpha,\beta)=\frac{(1+\beta)^3
				\beta^\alpha}{\Gamma(\alpha)}$ and $I_\nu(z)$ is the modified
				Bessel function. \apxonly{\nb{Before we had $(1+\beta)^1$ instead
				of cube.}}
	\end{proposition}
	\begin{proof}
		With this choice of $G_0$ the kernel $K(x,y)$ and $f_0(y)$ are given by~\eqref{eq:k_poisson}
		and~\eqref{eq:f0_negbin}, respectively. 
		Then 
	\begin{align*}
	S(x,x') 
	= & ~  \sum_{y \geq 0} f_0(y) K(x,y) K(x',y)\\
	= & ~ 	\sum_{y \geq 0}
	\binom{y+\alpha-1}{y}\left(\frac{\beta}{1+\beta}\right)^{\alpha} \left(\frac{1}{1+\beta}\right)^y 
	\frac{(1+\beta)^{2(y+\alpha+1)}}{\Gamma(y+\alpha)^2} (xx')^{y+\alpha-1} e^{-(1+\beta)
	(x+x')} \\
	\stepa{=} & ~ (1+\beta)^3{\beta^\alpha\over \Gamma(\alpha)} e^{-(1+\beta) (x+x')}
		\sum_{y \geq 0} 
		\frac{((1+\beta)xx')^{y+\alpha-1}}{y! \Gamma(y+\alpha)} \\
	\stepb{=} & ~ (1+\beta)^3\frac{\beta^\alpha}{\Gamma(\alpha)}  e^{-(1+\beta) (x+x')} 
((1+\beta)xx')^{\frac{\alpha-1}{2}}	
	I_{\alpha-1}(2 \sqrt{(1+\beta) xx'})\,,
			\end{align*}
	where in (a) we used ${y+\alpha-1\choose y} = {\Gamma(y+\alpha)\over y!
		\Gamma(\alpha)}$ and in (b) we applied the following
		identity~\cite[(8.445)]{GR}: for any $\nu \in \mreals$ and $z\in\mreals$ the
		modified Bessel functions satisfies
		$$ I_\nu(z) = \sum_{y=0}^\infty {1\over y! \Gamma(y+\nu + 1)} \left({z\over
		2}\right)^{\nu + 2y}\,.$$
	\end{proof}

	The easiest way to proceed would be to diagonalize the $S$ operator, i.e. to solve the equation 
	$ S \Gamma_\lambda = \lambda \Gamma_\lambda $. 
	Letting $\Gamma_\lambda(x) \equiv \phi_\lambda(\sqrt{bx})$, $b=2\sqrt{1+\beta}$, this equation is equivalent to 
		$$ {2 \sqrt{1+\beta} \beta^\alpha\over \Gamma(\alpha)} \int_{\reals_+} e^{-b(u^2+s^2)/4} (us)^{\alpha-1}
	I_{\alpha-1}(us) \phi_\lambda(s) s \, ds = \lambda \phi_\lambda(u)\,. $$
	Unfortunately, we were not able to solve this (except for $\alpha = 1$, in which case the solution found below
	is an actual eigenbasis; see \prettyref{eq:Seigen-poi} below.) Instead, we will find a collection of functions $\Gamma_i$ with
	the property that\footnote{Here and below, unless specified otherwise, inner products $(\cdot,\cdot)$ are with respect to
	$L_2(\mreals_+, \Leb)$.}
	\begin{equation}\label{eq:s_ort}
			(S \Gamma_i, \Gamma_j) = 0, \quad \forall i\neq j\,.
	\end{equation}		
	Our method to do so is based on the following observation. Suppose there is a decomposition
	\begin{equation}\label{eq:s_decomp}
		S(x,y) = w_1(x) w_1(y) \sum_n a_n f_n(x) f_n(y)\,,
	\end{equation}
	for some strictly positive weight function $w_1$, such that the system of functions $\{f_n\}$ is orthogonal in $L_2(\mreals_+, w)$.
	Then, the system of functions 
		\begin{equation}\label{eq:gam_recipee}
			\Gamma_n(x) \triangleq {w(x)\over w_1(x)} f_n(x)\,, 
		\end{equation}		
	satisfies~\eqref{eq:s_ort}. Indeed,
	\begin{equation}\label{eq:gc_0}
			(S \Gamma_i, \Gamma_j) = \sum_n  a_n \int dx \int dy \, w(x) w(y) f_i(x) f_n(x)
		f_j(y) f_n(y) = a_i \|f_i\|_{L_2(w)}^4 1\{i=j\}\,.
	\end{equation}

	To get the required decomposition~\eqref{eq:s_decomp} we recall the Hardy-Hille summation
	formula (cf.~\cite[8.976.1]{GR}): For all $x,y\in \mreals, \nu > -1$ and $|z|<1$ we have
		\begin{equation}
		\sum_{n\geq 0} n! \frac{L_n^{\nu}(x)L_n^{\nu}(y) z^n}{\Gamma(n+\nu+1)} =
		\frac{(xyz)^{-\nu/2}}{1-z} \exp\pth{-(x+y)\frac{z}{1-z}}
		I_\nu\pth{\frac{2\sqrt{xyz}}{1-z}}\,,
		\label{eq:HH}
		\end{equation}
	where $L_n^\nu$ are a family of generalized Laguerre polynomials, cf.~\cite[(8.970)]{GR}
	for the definition.
	We aim to apply this identity with $x \leftarrow \gamma_2 x$ and $y
	\leftarrow \gamma_2 x'$, where $\gamma_2 > 0$ is a parameter to be chosen shortly, to get
	an expression for $I_\nu(2\sqrt{(1+\beta)xx'})$. This forces the choice of $z$ as a
	solution of 
	\begin{equation}\label{eq:zgm_cons}
			{\sqrt{z}\over 1-z} \gamma_2 = \sqrt{1+\beta}\,. 
	\end{equation}		
	With this $z$ we get from~\eqref{eq:spoi_ddef} and~\eqref{eq:HH} the following expansion
	\begin{equation}
	S(x,x') = w_1(x) w_1(x') \sum_n a_n L_n^\nu(\gamma_2 x) L_n^{\nu}(\gamma_2 x')\,,
	\label{eq:Sexpand}
	\end{equation}
	where $w_1(x) \triangleq (\gamma_2 x)^{\nu} e^{-\gamma_3 x} $,
	$\gamma_3 \triangleq 1+\beta - \gamma_2 {z\over 1-z}$ and
	\begin{equation}\label{eq:gc_3}
		a_n = C(\alpha,\beta) (1+\beta)^{\nu\over 2} z^{2n+\nu\over 2} (1-z) \gamma_2^{-\nu} {n!\over \Gamma(n+\nu + 1)}\,.
	\end{equation}	

	We recall the orthogonality relation
	for Laguerre polynomials~\cite[22.2.12]{AS64}:
		\begin{equation}
		\int_{\reals_+} dx x^\nu e^{-x}L_m^\nu(x)L_n^\nu(x) = \frac{\Gamma(n+\nu+1)}{n!}  \indc{m=n}.
		\label{eq:Glag-ortho}
		\end{equation}
	Thus, setting $w(x) = (\gamma_2 x)^\nu e^{-\gamma_2 x}$ we get that $L_n^\nu(\gamma_2 x)$
	is an orthogonal system in $L_2(\mreals_+, w)$. Consequently, from~\eqref{eq:gam_recipee}
	we get our required system of functions
		\begin{equation}\label{eq:gamf_def}
			\Gamma_n(x) \eqdef e^{-\gamma_1 x} L_n^{\nu}(\gamma_2 x)\,, \quad \gamma_1
			\eqdef {\gamma_2\over 1-z} -1-\beta
		\end{equation}
	satisfying, in view of \prettyref{eq:gc_0}, \prettyref{eq:gc_3}, and \prettyref{eq:Glag-ortho},
		\begin{equation}
			(S\Gamma_n, \Gamma_m) = b_n 1\{n=m\}\,,  \quad ~b_n \eqdef C_2(\alpha,\beta)
		z^n {\Gamma(n+\alpha)\over n!} 
		\label{eq:SGammakx} 
		\end{equation}		
	and 
	\[
	C_2(\alpha,\beta) \triangleq C(\alpha,\beta)  (1-z) ((1+\beta) z)^{\alpha -1 \over 2}
	\gamma_2^{-\alpha-1}
	\]
	with $C(\alpha,\beta)=\frac{(1+\beta)^3
				\beta^\alpha}{\Gamma(\alpha)}$ given in \prettyref{prop:Skernel-poi}.
	\apxonly{\nb{this value of $b_n$ coincides with our previous
	estimate.}}
We note that in the special case of $\alpha=1$, we have $\nu=0$ and $\Gamma_n(x) = e^{-\gamma_1 x} L_n(2\gamma_1 x)$ with $L_n$ being the usual Laguerre polynomials. 
Using \prettyref{eq:Sexpand} and \prettyref{eq:Glag-ortho}, one can verify that $S(x,x') = \sum_{n\geq 0} a_n \Gamma_n(x)\Gamma_n(x')$ and 
\begin{equation}
S \Gamma_ k  =  \frac{a_k}{\gamma_1} \Gamma_k.
\label{eq:Seigen-poi}
\end{equation}
In other words, for $\alpha=1$, $\{\Gamma_k\}$ are indeed the eigenbasis of the operator $S$; however, this is not true when $\alpha \neq 1$ (which will be the case of proving the compactly supported case of \prettyref{th:poisson}).

	To proceed further, we need to recall the following properties of Laguerre polynomials:
	\begin{itemize}
		\item Exponential growth \cite[22.14.13]{AS64}: for all degrees $n$ and all $\nu,x\geq 0$,
		\begin{equation}
		|L_n^\nu(x)| \leq e^{x/2} \binom{n+\nu}{n}\triangleq e^{x/2} \frac{\Gamma(n+\nu+1)}{n!\Gamma(\nu+1)}.
		\label{eq:Glag-exp}
		\end{equation}
		
		\item Recurrence (cf.~\cite[22.8 and 22.7]{AS64}):
		\begin{align}
		x \frac{d}{dx} L_n^\nu(x)= & ~ 	n L_n^\nu(x) - (n+\nu) L_{n-1}^\nu(x) 		\label{eq:Glag-diff}\\
		x L_n^\nu(x)= & ~ 	(2n+\nu+1) L_n(x) - (n+1) L_{n+1}(x)-(n+\nu) L_{n-1}(x) 		\label{eq:Glag-3recur}
		\end{align}

	\end{itemize}

	To complete the construction we still need to choose $\gamma_2>0$. We do so in a way that
	permits us to prove the $L_\infty$-boundedness of
	$\Gamma_n(x)$ functions by way of~\eqref{eq:Glag-exp}. That is, we set 
		$$ \gamma_1 = \gamma_2 / 2\,. $$
	Together with~\eqref{eq:zgm_cons} and~\eqref{eq:gamf_def} this forces the following choice of $\gamma_1,\gamma_2,\gamma_3,z$:
	\begin{equation}\label{eq:gc_2}
			z = (\sqrt{1+\beta} - \sqrt{\beta})^2, \quad \gamma_1 = \gamma_3 = \sqrt{\beta (1+\beta)}, \quad \gamma_2 =
		2\sqrt{\beta(1+\beta)}\,.
	\end{equation}		
	We note that in order for~\eqref{eq:HH} to hold we also need to make sure $|z|<1$, which
	is indeed the case for all $\beta>0$.

	Similarly to how we defined $S = K^* K$ in \prettyref{eq:KtoS-poi},  we also define $S_1 = K_1^* K_1$. We summarize
	properties of the constructed set of functions in the following Lemma.

		\begin{lemma} The system of functions~\eqref{eq:gamf_def} with choice~\eqref{eq:gc_2} satisfies
			\begin{align}
			\|\Gamma_k\|_\infty  \le & ~ {k+\alpha \choose k} \label{eq:gammakinf}\\
			(S \Gamma_k,\Gamma_k)= & ~b_k \label{eq:SGammakx_2}\\
			(S_1 \Gamma_k,\Gamma_k)= & \frac{b_k }{4 (1+\beta)^2} \sth{\alpha^2 
				+ (k+1)(k+\alpha) z + (k+\alpha-1) k \frac{1}{z} }
			\label{eq:S1Gammakx}\,,
			\end{align}
			with $b_k$ defined in~\eqref{eq:SGammakx}.
			Furthermore,
			\begin{equation}
			(S\Gamma_k,\Gamma_j)=  0, \quad k \neq j
			\label{eq:SGamma-orthox}
			\end{equation}
			and
			\begin{equation}
			(S_1 \Gamma_k,\Gamma_j)=  0 \qquad \forall |k-j|\ge 3
			\label{eq:S1Gamma-orthox}
			\end{equation}
		\end{lemma}
		\begin{proof}
			We note that~\eqref{eq:SGammakx_2} and~\eqref{eq:SGamma-orthox} have been
			established in~\eqref{eq:SGammakx} already, and \eqref{eq:gammakinf} in \eqref{eq:Glag-exp}. 

			To prove identities involving $S_1$, note that using \prettyref{prop:KK1x} we get
			\begin{equation}
			(S_1\Gamma_k,\Gamma_j) =   {1\over (1+\beta)^2}(K(x \Gamma_k'), K(x \Gamma_j'))_{L_2(f_0)} 
				=  {1\over (1+\beta)^2}(S(x \Gamma_k'), x \Gamma_j')\,.
			\label{eq:S1Gammakjx}
			\end{equation}
			Denoting for convenience $\gamma\equiv\gamma_1 = \gamma_2/2$ such that $\Gamma_k(x)=e^{-\gamma x} L_n^\nu(2\gamma x)$
			and recalling that
			$\nu=\alpha-1$ we have from the recurrence \prettyref{eq:Glag-diff}--\prettyref{eq:Glag-3recur} of generalized Laguerre polynomials:
			\begin{align*}
			x \Gamma_k'(x)
			= & ~ e^{- \gamma x} \sth{-\gamma x L_k^{\nu}( 2 \gamma x) + 2 \gamma x (L_k^{\nu})'(2\gamma x)} \\
			= & ~ e^{- \gamma
			x} \sth{- \frac{1}{2} \pth{(2k+\nu+1)L_k^{\nu}-(k+1)L_{k+1}^{\nu}-(k+\nu)L_{k-1}^{\nu}}  + (kL_{k}-(k+\nu)L_{k-1}^{\nu}) }(2\gamma x) \\
			= & ~ e^{- \gamma x} \sth{- \frac{\alpha}{2} L_k^{\nu} + \frac{k+1}{2} L_{k+1}^{\nu} -\frac{k+\nu}{2} L_{k-1}^{\nu}  }(2\gamma x) \\
			= & ~ - \frac{\alpha}{2} \Gamma_k(x) + \frac{k+1}{2} \Gamma_{k+1}(x) -\frac{k+\nu}{2} \Gamma_{k-1}(x). \numberthis \label{eq:xGammakprime}
			\end{align*}
		
			If $|k-j|\geq 3$, \prettyref{eq:S1Gamma-orthox} follows from \prettyref{eq:S1Gammakjx} and the orthogonality property \prettyref{eq:SGammakx}.
			Finally, using \prettyref{eq:xGammakprime} and \prettyref{eq:SGammakx}, \prettyref{eq:S1Gammakx} follows from
			\begin{align*}
				& ~ (S_1\Gamma_k,\Gamma_k) = {1\over (1+\beta)^2}(S(x \Gamma_k'), x \Gamma_k') \\
			= & ~ \frac{1}{4 (1+\beta)^2} \sth{\alpha^2 b_k + (k+1)^2 b_{k+1}+ (k+\alpha-1)^2 b_{k-1} }	\\
			= & ~ \frac{C_2(\alpha,\beta) z^k }{4 (1+\beta)^2} \sth{\alpha^2 \frac{\Gamma(k+\alpha)}{k!}
			+ (k+1)^2 \frac{\Gamma(k+\alpha+1)}{(k+1)!} z + (k+\alpha-1)^2 \frac{\Gamma(k+\alpha-1)}{(k-1)!} \frac{1}{z} 		}	\\
			= & ~ {b_{k} \over 4(1+\beta)^2} 
			\sth{\alpha^2 
				+ (k+1)(k+\alpha) z + (k+\alpha-1) k \frac{1}{z} }. 
			\end{align*} 
		\end{proof}

		We are now in a position to prove Lemma~\ref{lem:poisson_main}.

		\begin{proof}[Proof of Lemma~\ref{lem:poisson_main}]

		Fix $m$ and select
				$$ r_k = {1\over \sqrt{(S_1 \Gamma_k, \Gamma_k)}} \Gamma_k, \qquad k \in \calK = \{m,m+3,\ldots,4m\}\,.$$
			This ensures that conditions~\eqref{eq:lpm_1} and~\eqref{eq:lpm_2} are satisfied.
			Recall that $z = \frac{1}{(\sqrt{\beta+1}+\sqrt{\beta})^2}$ and the assumption that $\beta \geq 2$.		
			We have
			$$ m \leq k \leq 4 m, \quad \frac{1}{6\beta} \leq z \leq {1\over 4\beta} \leq \frac{1}{8}\,. $$
			Denote $\delta = \alpha/\beta$.			
			To show~\eqref{eq:lpm_3} notice that from \eqref{eq:SGammakx} and \eqref{eq:S1Gammakx} we have for any $k\in
			\calK$
			\[ 
			\|K r_k\|_{L^2(f_0)}^2 = 
			(S r_k,r_k) = 
			{(S \Gamma_k, \Gamma_k)\over (S_1 \Gamma_k, \Gamma_k)} 
			= \frac{4(1+\beta)^2}{\alpha^2 
				+ (k+1)(k+\alpha) z + (k+\alpha-1) \frac{k}{z} }
				\asymp 
				{1\over \delta m + \delta^2}\,,
				\]
			since  $\alpha^2 
					+ (k+1)(k+\alpha) z + (k+\alpha-1) k \frac{1}{z}  \asymp \alpha^2 + m \alpha \beta$, using the assumption of $m\leq \alpha$.
					
			To show \prettyref{eq:lpm_4}, we use~\eqref{eq:gammakinf} to get
			\begin{equation}\label{eq:gc_7}
					\max_{k\in\calK} \|r_k\|_\infty \lesssim \max_{k\in\calK} \sqrt{1\over \delta m  b_k} {k+\alpha\choose
				k}\,.
			\end{equation}		
			Using \eqref{eq:SGammakx} and $\gamma_2 = 2\sqrt{\beta(1+\beta)}$, we have
				\begin{align*} 
				b_k &= (1+\beta)^3(1-z) \cdot \beta^\alpha \cdot (z(1+\beta))^{\alpha -1 \over 2}  \cdot
			\gamma_2^{-\alpha-1} \cdot z^k {k+\alpha-1\choose k}\\
				& = A B^{\alpha/2} z^k {k+\alpha-1\choose k}
			\end{align*}
			where  $B \triangleq \frac{z \beta^2 (1+\beta)}{\gamma_2^2} = \frac{z\beta}{4} \geq  \frac{1}{24}$ and 
			$A \triangleq \frac{(1+\beta)^3(1-z)}{2\sqrt{z\beta}} \geq \frac{1}{2}$ for all $\beta \geq 2$. 
			Thus we get 
			$b_k \geq \frac{1}{2} 5^{-\alpha} z^k$ and (using $k\leq 4m$)
			\[
					{k+\alpha\choose k}^2 b_k^{-1} \le 2 \cdot 5^\alpha (6 \beta)^k 4^{k+\alpha} \leq \exp(C' (\alpha + m\log \beta))
			\]
			for some absolute constant $C'$.
			Thus, from~\eqref{eq:gc_7} we have shown that 
				$$\max_{k\in\calK} \|r_k\|_\infty  \le \sqrt{\frac{\beta}{\alpha}} e^{C (m \log \beta + \alpha)}\,,$$
				for some absolute constant $C$. This completes the proof of~\eqref{eq:lpm_4}.
		\end{proof}

	\begin{proof}[Proof of Lemma~\ref{lem:poisson_exp}]
	For $\alpha=1$ and fixed $\beta>0$ (not dependent on $m$) we select functions $r_j$ as in the proof
	of Lemma~\ref{lem:poisson_main}. This time, however, we have $C(\alpha,\beta) \asymp
	C_2(\alpha,\beta) \asymp z \asymp \gamma_1 \asymp \gamma_2 \asymp 1$, $b_k \asymp z^k$. This
	implies via~\eqref{eq:SGammakx_2}--\eqref{eq:S1Gammakx} that
	$$ \|K r_j\|_{L_2(f_0)}^2 = {(S\Gamma_k, \Gamma_k)\over (S_1\Gamma_k, \Gamma_k)} \asymp
	{1\over m^2}\,,$$
	since $(\alpha^2 + (k+1)(k+\alpha)z + (k+\alpha-1)k{1\over z}) \asymp m^2$. This
	proves~\eqref{eq:lpm_3e}. To show~\eqref{eq:lpm_4e} we notice that
	from~\eqref{eq:gammakinf} and $\alpha=1$:
		$$ \|r_k\|_\infty = {\|\Gamma_k\|_\infty\over \sqrt{(S_1\Gamma_k,\Gamma_k)}} 
			\lesssim {m\over \sqrt{z^{4m} m^2}} = z^{-{2m}}\,.$$ 
	Since $z < 1$ we get the estimate~\eqref{eq:lpm_4e}.
\end{proof}

\section{Regret bound for the Robbins estimator}
\label{app:robbins}

In this appendix we prove the upper bound part of \prettyref{th:poisson} for the Poisson model by analyzing 
the Robbins estimator $\tilde \theta^n$ in \prettyref{eq:robbins}.
We first consider the compact supported case and show that for any prior $G$ on $[0,h]$, 
	\begin{equation}\label{eq:poi_ub}
		\Expect_G\qth{\|\tilde \theta^n - \theta^n\|^2}
	- n \cdot \mmse(G) \le c_2 \left({\log n \over \log \log n}\right)^2\,,
\end{equation}
for some constant $c_2=c_2(h)$. 
To this end, we mostly follow the method of~\cite{BGR13} fixing a
mistake there\footnote{In the first display equation in the proof of~\cite[Theorem 1]{BGR13},
the last identity does not hold.} and adapting the proof to fixed sample size $n$ (as
opposed to $\Poi(n)$ in~\cite{BGR13}). First, we consider the case of bounded support. 
For any fixed prior $G$ and its induced
Poisson mixture $f$, recall the Bayes optimal estimator in \prettyref{eq:bayes-poi}, namely
\begin{equation}\label{eq:abc_3}
	\hat \theta_G(y) = (y+1){f(y+1)\over f(y)}\,.
\end{equation}
Then the total regret of Robbins' estimator under the prior $G$ is given by
\begin{align*} 
R(G,\tilde \theta^n) 
&\triangleq \sum_{i=1}^n \pth{\Expect[(\tilde \theta_i-\theta_i)^2] - \Expect[(\hat \theta_G(Y_i)-\theta_i)^2]} \stepa{=} \sum_{i=1}^n \Expect[(\tilde \theta_i-  \hat \theta_G(Y_i))^2]  \\
&\stepb{=} \sum_{y=0}^\infty \EE\left[N(y) (y+1)^2\left( {N(y+1)\over N(y)} -
{f(y+1)\over f(y)}\right)^2 1\{N(y)>0\} \right]\\
&= \sum_{y=0}^\infty \EE\left[ {1\{N(y)>0\}(y+1)^2\over N(y)} \left( N(y+1) -
{f(y+1) N(y)\over f(y)}\right)^2  \right]
\end{align*}
where (a) follows from $\Expect[\theta_i|Y_i]=\hat\theta_G(Y_i)$;
in (b) we recall the definition of $N(y) = \sum_{i=1}^n \indc{Y_i=y}$ from \prettyref{eq:robbins}.

Conditioning on $N(y)$ we have that $N(y+1) \sim \Bino(n-N(y), {f(y+1)\over 1-f(y)})$.
Thus, by the (conditional) bias-variance decomposition we have (denoting 
temporarily $q \triangleq {f(y+1)\over 1-f(y)}$):
\begin{align*} 
&~ \EE\left[\left( N(y+1) - {f(y+1)N(y)\over f(y)}\right)^2 \middle | N(y)\right] 
=(n-N(y))q(1-q) + \pth{(n-N(y))q - {f(y+1)\over f(y)} N(y)}^2 \\
	\le &~ n {f(y+1)\over 1-f(y)} + \left({f(y+1)\over (1-f(y))f(y)}\right)^2
	(nf(y)-N(y))^2\,.
\end{align*}
We now notice a simple fact about the Poisson mixture density $f(y)$:  Since for every $x>0$, ${x^y e^{-x}\over y!} \le {y^y
e^{-y}\over y!} \le {1\over \sqrt{2\pi y}}$ (Stirling's), we have
$$ f(y) \le {1\over \sqrt{2\pi y}}\, \qquad \forall y>0\,.$$
Thus, in all terms except $y=0$ we have ${1\over 1-f(y)} \lesssim 1$ and 
we can write
\begin{equation}\label{eq:abc_2}
	R(G,\tilde \theta^n) \lesssim I_0 + \sum_{y=1}^\infty I_1(y) + I_2(y) 
\end{equation}
where
\begin{align*} I_0 &\triangleq {f(1)\over 1-f(0)} n v_1(n,f(0)) + \left({f(1)\over (1-f(0))f(0)}\right)^2
	v_2(n,f(0))\\
   I_1(y) &\triangleq (y+1)^2 n f(y+1) v_1(n, f(y))\\
   I_2(y) &\triangleq \left((y+1){f(y+1)\over f(y)}\right)^2  v_2(n, f(y))\,, 
\end{align*}
and the two functions $v_1,v_2$ are defined in the next lemma (whose proof is given at the end of this appendix):
\begin{lemma}\label{lem:v1v2} Let $B\sim\Bino(n,p)$ and define 
	$$ v_1(n,p) \triangleq \EE[B^{-1}1\{B>0\}]\,, \quad v_2(n,p) \triangleq \EE[B^{-1}1\{B>0\} (B-np)^2]\,.$$
	There exist absolute constants $c_1,c_2>0$ such that for all $n\ge1,p\in[0,1]$ we have
	$$ v_1(n,p) \le c_1 \min\pth{np, {1\over np}}\le c_1, \quad v_2(n,p) \le c_2 \min(1, np) \le
	c_2\,.$$
\end{lemma}

Assuming this lemma, we proceed to bounding the terms in~\eqref{eq:abc_2}. 
For $I_0$, 
we obtain by bounding $v_1(n, f(0)) \lesssim {1\over n f(0)}$ and
$v_2 \lesssim 1$ that
\begin{equation}\label{eq:abc_4}
	I_0 \lesssim {f(1)\over (1-f(0))f(0)} +  \left({f(1)\over (1-f(0))f(0)}\right)^2\,.
\end{equation}
Now, notice that $f(1) \le 1-f(0)$ and thus ${f(1)\over (1-f(0))f(0)} \le {1\over f(0)}$. To get
another bound, we need to make a key observation: 
for every $y\geq0$, $\hat \theta_G(y) = \EE_G[\theta|Y=y]$ defined in~\eqref{eq:abc_3} 
 belongs to the interval $[0,h]$ (since $G$ is supported on $[0,h]$); in other words,
\begin{equation}\label{eq:abc_5}
	(y+1) {f(y+1)\over f(y)} \le h \qquad \forall y\ge 0\,.
\end{equation}
In particular, applying this with $y=0$ to bound ${f(1)\over f(0)} \le h$ we get, overall, that
$$ {f(1)\over (1-f(0))f(0)} \le \min\left({1\over f(0)}, {h\over 1-f(0)}\right) \lesssim \max(1,h)\,.$$
From~\eqref{eq:abc_4}, thus, we get 
\begin{equation}\label{eq:abc_7}
	I_0 \lesssim \max(1,h)\,.
\end{equation}

Next, we bound $I_1(y)$ via \prettyref{lem:v1v2} and~\eqref{eq:abc_5} as
\begin{align} I_1(y) &\lesssim (y+1)^2 n f(y+1) \min\pth{nf(y), {1\over nf(y)}}\nonumber\\
	&= (y+1)\left( (y+1) {f(y+1)\over
f(y)}\right) \min\pth{(nf(y))^2, 1} \le h(y+1) \min\pth{1, (nf(y))^2}\,.\label{eq:abc_6}
\end{align}

To proceed, we need another lemma (also proved at the end of this appendix):
\begin{lemma}\label{lem:sumfy_bd}
Let $f$ be a Poisson mixture with mixing distribution $G$.
\begin{itemize}
	\item Suppose $G$ is supported on $[0,h]$. Then for
some constants $c = c(h)$ and $y_0 = \max(2h, c {\log n \over \log \log n})$,
\begin{align} \sum_{y\ge y_0} f(y) &\le {1\over n}\,,\label{eq:sf_1}\\
		\sum_{y> y_0} f(y)^2 (y+1) & \le {4h\over n^2}  \label{eq:sf_2}
	\end{align}
	\item Suppose 
	$G$ satisfies $G[X>x] \le ae^{-bx}$ for all $x>0$. Then
	for some constant $c_2=c_2(a,b)$ and $y_1 = c_2 \log n$,
\begin{align} \sum_{y\ge y_1} f(y) &\le {1\over n}\,,\label{eq:sf_1b}\\
		\sum_{y> y_1} f(y)^2 (y+1) & \le {1\over n^2}  \label{eq:sf_2b}
	\end{align}
\end{itemize}
\end{lemma}

Thus, choosing $y_0$ as above we can sum~\eqref{eq:abc_6} to get (in the sequel all constants depend on $h$) 
\begin{align} \sum_{y\ge 1} I_1(y) & \lesssim \sum_{y\le y_0} (y+1) + n^2 \sum_{y> y_0} f(y)^2 (y+1) \nonumber \\
			& \lesssim y_0^2 + 1 \asymp \left({\log n \over \log \log n}\right)^2\,. \label{eq:I1final}
\end{align}
Similarly, we have
\begin{align} \sum_{y\ge 1} I_2(y) &\stepa{\le}h^2 \sum_{y\ge 1} v_2(n, f(y)) \nonumber\\ 
			&\stepb{\lesssim} h^2 \sum_{y\ge 1} \min(1, nf(y))\label{eq:abc_8}\\
			&\stepc{\le} \sum_{y<y_0} 1 + n \sum_{y\ge y_0} f(y) \le y_0 + 1 \asymp
			{\log n \over \log \log n}\,, \label{eq:I2final}
\end{align}			
where (a) is by~\eqref{eq:abc_5}, (b) is Lemma~\ref{lem:v1v2}, (c) is by taking $y_0$ as in
Lemma~\ref{lem:sumfy_bd} and applying~\eqref{eq:sf_2}.

Substituting \prettyref{eq:abc_7},  \prettyref{eq:I1final}, and \prettyref{eq:I2final} into \prettyref{eq:abc_2}, we obtain
$$ R(G, \tilde \theta^n) \le I_0 + \sum_{y\geq 1} I_1(y) + I_2(y) \lesssim \left({\log n \over \log \log
n}\right)^2\,,$$
completing the proof of~\eqref{eq:poi_ub}.

Next, we proceed to the subexponential case. Fix a prior $G \in \SubE(s)$, namely, 
$G((t,\infty)) \leq 2e^{-t/s}$ for all $t>0$. 
In the rest of the proof all constants depend on $s$.
We first show that it is possible to
replace $G$ with a truncated prior $G'[X\in \cdot] = G[X\in \cdot | X\le c \log n]$ for a
suitably large $c=c(s)>0$. Indeed, observe that the Robbins estimator for any coordinate $j$ satisfies:
\begin{equation}\label{eq:rb_1}
	\tilde \theta_j = (Y_j+1) {N(Y_j+1)\over N(Y_j)} \le (Y_j+1) n\,.
\end{equation}
This is because, crucially, $N(Y_j)\geq 1$ by definition.
Since 
$\Expect[Y_j^4] = \int G(d\theta) (\theta ^4+6 \theta ^3+7 \theta ^2+\theta) \lesssim 1$, 
we have 
\begin{equation}
\EE[\tilde \theta_j^4] \leq n^4 \Expect[(Y_j+1)^4] \lesssim n^4.
\label{eq:robbins-fourthmoment}
\end{equation} 
Next, since the $G$ is subexponential, we
can select $c$ so large that 
$$ \epsilon = G[X>c \log n]  \le {1\over n^{10}}\,.$$

Now, denote as in the proof of Lemma~\ref{lmm:truncation} the event $E = \{\theta_j \le c \log
n, j=1,\ldots,n\}$. Recall that $\EE_G$ denotes expectation with respect to $\theta_i \simiid G$ and $Y_i \sim \Poi(\theta_i)$.
Then 
\begin{align} {1\over n} R(G, \tilde \theta^n) &=\EE_G[(\tilde \theta_1 - \theta_1)^2] -
\mmse(G)\nonumber\\
	&\le \EE_G[(\tilde \theta_1 - \theta_1)^2 | E] - \mmse(G') + \mmse(G')-\mmse(G) +
	\EE_G [(\tilde \theta_1 - \theta_1)^2 1_{E^c}]\nonumber\\
	& = {1\over n} R(G', \tilde \theta^n) + \mmse(G')-\mmse(G) +
	\EE_G [(\tilde \theta_1 - \theta_1)^2 1_{E^c}]\,, \label{eq:rb_2}
\end{align}
where the first identity follows from symmetry.
For the last term in~\eqref{eq:rb_2} we have from Cauchy-Schwarz and $\PP[E^c] \le n \epsilon$:
	$$ \EE_G [(\tilde \theta_1 - \theta_1)^2 1_{E^c}] \le \sqrt{n\epsilon} \sqrt{\EE_G [(\tilde
	\theta_1 - \theta_1)^4} \lesssim \sqrt{n\epsilon} n^2\,,$$
where in the last step we used~\eqref{eq:robbins-fourthmoment} and $\EE_G[\theta_1^4] \lesssim 1$ (since $G$ is
subexponential).

For the second term, as in \eqref{eq:mmse-truncate} we use $\mmse(G') \le {1\over 1-\epsilon}
\mmse(G)$ so that
	$$ \mmse(G') - \mmse(G) \le {\epsilon\over 1-\epsilon} \mmse(G) \lesssim \epsilon
	\,.$$
Altogether, we get from~\eqref{eq:rb_2} that 
\begin{align*}  R(G, \tilde \theta^n) &\le R(G', \tilde \theta^n) + O(1/n)\,,
\end{align*}
Thus, it is sufficient to analyze the regret of Robbins under the prior $G'$. Note that $G'$ satisfies the assumptions in the
second part of Lemma~\ref{lem:sumfy_bd} with constants $a,b$ depending on $s$ only. According to~\eqref{eq:abc_2} we need to bound $I_0$,
$\sum_{y\ge 1} I_1(y)$ and $\sum_{y\ge 2} I_2(y)$. We set $h=c\log n$ (which is the support of
$G'$) and get from~\eqref{eq:abc_7} the bound $I_0 \lesssim \log n$. 
For $I_1$ we have from~\eqref{eq:abc_6} the following estimate
$$ \sum_{y\ge 1} I_1(y) \lesssim h \sum_{y\ge 1} (y+1)  \min(1, nf(y))^2 \le h \left(\sum_{y=1}^{y_1} (y+1)
+ n^2 \sum_{y > y_1} (y+1) f^2(y)\right)\,,$$
where $y_1\asymp \log n$ is defined in Lemma~\ref{lem:sumfy_bd}. Then the first sum is proportional to $\log^2 n$ and the
second, in view of ~\eqref{eq:sf_2b}, is at most ${1\over n^2}$. Overall we get
$$ \sum_{y\ge 1} I_1(y) \lesssim h (y_1^2 + 1) \asymp \log^3 n\,.$$
Similarly, from~\eqref{eq:abc_8} we have 
$$ \sum_{y \ge 1} I_2(y) \lesssim h^2 \sum_{y\ge 1} \min(1, nf(y)) \le h^2\left( \sum_{y\le y_1} 1 +
n \sum_{y>y_1}f(y)\right)\,.$$
The first sum equals $y_1$ and the second is bounded by ${1\over n}$ (via~\eqref{eq:sf_1b}), so we
get
$$ \sum_{y \ge 1} I_2(y) \lesssim h^2 (y_1 + 1) \asymp \log^3 n\,.$$
Altogether, we obtain the claimed $O(\log^3 n)$ regret bound for the subexponential class.

\begin{remark} In the analysis of Robbins estimator, if we denote by $y$ the effective support of
the Poisson mixture (i.e.~the $1-{1\over n^{10}}$ quantile) and by $h$ the effective
support of the mixing distribution, then the total regret is at most $O(h y (h+y))$. In the compact
support case $h\asymp 1$ and $y \asymp {\log n \over \log \log n}$, while in the subexponential
case $h\asymp y \asymp \log n$.
This explains the rates in \prettyref{eq:poi_lb} and \prettyref{eq:poi_lb2}, respectively.
\end{remark}

\begin{proof}[Proof of Lemma~\ref{lem:v1v2}]
	Note that $B^{-1}1\{B>0\} \le {2\over B+1}1\{B>0\}$. Furthermore, we have
	\begin{align*} \EE[(B+1)^{-1} 1\{B>0\}] &= \sum_{k=1}^n {n!\over (k+1)! (n-k)!} p^k (1-p)^{n-k} \\
			&= {1\over p(n+1)} \sum_{k=1}^n {n+1\choose k+1} p^{k+1} (1-p)^{n-k}\\
			&\le {1\over pn} \PP[\Bino(n+1,p) \ge 2] \leq \frac{1}{np}\,.
	\end{align*}
	For the case of $np\leq1$, 
		$$ 
		\Expect[B^{-1}1\{B>0\}]\leq \Prob[B\geq 1] \leq \sum_{k=1}^n \frac{(np)^k}{k!} \leq np  \cdot \Prob[\Poi(np)\geq 0] \leq np.$$
	Overall, we have
	\begin{equation}\label{eq:abc_1}
			\EE[(B+1)^{-1} 1\{B>0\}] \le \min\pth{np, {1\over np}}\,,
	\end{equation}		
	and the desired bound $v_1(n,p) \equiv \Expect[B^{-1}1\{B>0\}] \leq 2 \min(np,\frac{1}{np})$.

	For $v_2$, we again bound $B^{-1}1\{B>0\} \le {2\over B+1}1\{B>0\}$, and also notice that
	$(B-np)^2 =  (B+1 - (n+1)p - (1-p))^2 \le 2 (B+1 - (n+1)p)^2 + 2 (1-p)^2$. This implies 
	\begin{align*} v_2(n,p) &\le 2\EE[(B+1)^{-1} 1\{B>0\} \{((n+1)p - (B+1))^2 + (1-p)^2\}]
	\end{align*}
	Notice that by~\eqref{eq:abc_1} expectation of the second quadratic term can be bounded as
	$\lesssim (1-p)^2$. Thus, we focus on the first term, for which we have the following
	expression in terms of $\tilde B \sim \Bino(n+1,p)$:
	\begin{align*} 
	& \lefteqn{\EE[(B+1)^{-1} 1\{B>0\} \{((n+1)p - (B+1))^2}\\
	&= \sum_{k=1}^n {n!\over (k+1)! (n-k)!} p^k (1-p)^{n-k} (k+1 -
	(n+1)p)^2 \\
		&= {1\over (n+1)p} \EE[1\{2\leq\tilde B\le n+1\} ((n+1)p-\tilde B)^2]\\
		&\le {1\over (n+1)p} \Var[\tilde B] = 1-p
	\end{align*}
	This proves $v_2(n,p) \le 2$. On the other hand, suppose $np<1$. Then for any $B\ge 1$
	we have $(B-np)^2 \le B^2$ and, consequently, 
	$$  v_2(n,p) \le \EE[1\{B>0\} B] \le np\,.$$
	And, hence, we have $v_2 \le 2\min(1,np)$.
\end{proof}

\begin{proof}[Proof of Lemma~\ref{lem:sumfy_bd}]
First, let us consider the compactly supported case.
	For any $y>h>x$ we have ${x^y e^{-x}\over y!} \le \bar f(y) \eqdef {h^y e^{-h}\over y!}$, implying that
	$f(y) = \EE_{X\lambda_2}[{X^y e^{-X}\over y!}] \le \bar f(y)$. 
	Note that $\bar f(y)$ is monotonically decreasing in $y$ for $y>h$. Furthermore, we have
	for any $y_0 \geq 2h$ that 
	\begin{align} 
	\sum_{y\ge y_0} \bar f(y) &= {h^{y_0} e^{-h}\over y_0!}
		\sum_{m\ge 0} {y_0!\over (y_0+m)!} h^m\nonumber\\
		&\le \bar f(y_0) \sum_{m\ge 0} \left({h\over y_0}\right)^m \le 2 \bar
		f(y_0).\label{eq:sff_0}
	\end{align}
	Furthermore, we have 
	\begin{align*} \sum_{y > y_0} \bar f(y)^2 (y+1) &\stepa{\le} 2 \sum_{y > y_0} y \bar f(y)^2\\
			&\stepb{=} 2h \sum_{y > y_0} \bar f(y-1) \bar f(y)^2\\
			&\stepc{\le} 2h \bar f(y_0) \sum_{y > y_0} \bar f(y)\\
			&\stepd{\le} 4h \bar f(y_0)^2\,, 
	\end{align*}
	where in (a) we used the fact that $y\geq 1$, in (b) the identity $y\bar f(y) = h \bar
	f(y-1)$, in (c) the monotonicity of $\bar f$, and in (d) we applied~\eqref{eq:sff_0}.
	Thus, selecting $y_0 > c(h) {\log n \over \log \log n}$ yields $\bar f(y_0)  \le {1\over
	n}$ and completes the proof of~\eqref{eq:sf_1} and~\eqref{eq:sf_2}.

For the subexponential case, let us denote by $\bar F(x) = G[X>x]$ and recall $\bar F(x) \le a  e^{-bx}$. Then from integration by parts we have
\begin{align*} f(y) = \EE_{X\lambda_2}\left[{X^y e^{-X}\over y!}\right] &= {1\over y!}\int_0^\infty  (y-x)x^{y-1} e^{-x}
\bar F(x) dx\\
	&\le {1\over y!}\int_0^\infty  y x^{y-1} e^{-x}
\bar F(x) dx\\
	&\le {a\over (y-1)!}\int_0^\infty  x^{y-1} e^{-(1+b)x} dx\\
	&= a (1+b)^{-y} \triangleq \tilde f(y)\,.
\end{align*}
The estimates~\eqref{eq:sf_1b} and~\eqref{eq:sf_2b} then easily follow from the properties of
geometric distribution $\tilde f$.
\end{proof}

\section{Proofs for the compound setting}
\label{app:comp}

\begin{proof}[Proof of \prettyref{prop:regret-comp}]
	As noted in \cite{greenshtein2009asymptotic}, the first inequality of \prettyref{eq:cc_1} follows from $R_{\oracle}(G_{\theta^n}) \le n\mmse(G_{\theta^n})$ since $f(Y^n) = (f_1(Y_1),\ldots,f_1(Y_n))$ is permutation-invariant. 
For the second, by the concavity of $G\mapsto\mmse(G)$ (see \cite[Corollary 1]{mmse.functional.IT}), we have
\begin{equation}
\EE_{\theta^n \simiid G} \mmse(G_{\theta^n}) \le \mmse(G)\,,
\label{eq:mmse-emp}
\end{equation}
and, therefore, 
$$ \EE_G[\|\theta^n - \hat \theta^n(Y^n)\|^2] - n \mmse(G) \le \EE_G[\|\theta^n - \hat
\theta^n(Y^n)\|^2 - n \mmse(G_{\theta^n})].
$$
Taking $\inf_{\hat \theta^n} \sup_G$ on both sides proves the second inequality in~\eqref{eq:cc_1}.
\end{proof}

\begin{proof}[Proof of \prettyref{prop:comp-subg}]
	The proof follows the proof strategy for the second inequality in \prettyref{eq:cc_1} and the fact that the empirical distribution drawn from a subgaussian distribution is subgaussian (with multiplicative adjustment of constants)	with probability $1-1/\poly(n)$.

	Fix $s>0$ and $G\in \SubG(s)$. 
	Let $\theta_1,\ldots,\theta_n \iiddistr G$ and denote the empirical distribution $G_{\theta^n} = \frac{1}{n}\sum_{i=1}^n \delta_{\theta_i}$.
	We first show that there exists an absolute constant $C_0$ such that 
	\begin{equation}
	\prob{G_{\theta^n} \in \SubG(C_0 s)} \geq 1 - n^{-4}.
	\label{eq:emp-subg}
	\end{equation}
	Indeed, by the equivalent characterization of subgaussian constant (see \cite[Proposition 2.5.2]{vershynin2018high}), it suffices to show 
	\begin{equation}
	\prob{\frac{1}{n} \sum_{i=1}^n \exp\pth{\frac{\theta_i^2}{C_1 s}} > 2} \leq n^{-4}.
	\label{eq:emp-subg1}
	\end{equation}
	for some absolute constant $C_1$. Since $\prob{|\theta_i| \geq t} \leq 2 e^{-t^2/(2s)}$ by assumption, there exists 
	$C_1$ such that iid random variables $X_i \eqdef\exp\pth{\frac{\theta_i^2}{C_1 s}}$
	satisfy $\expect{X} \leq 1.1$ and $\expect{X_i^8} \leq 1.2$. From Rosenthal's
	inequality~\cite[Theorem 3]{rosenthal1970subspaces} we obtain $\EE[\left(\sum_i X_i-\EE[X]\right)^8] \lesssim
	n^{4}$, and then from Markov's inequality $\PP[{1\over n}\sum_i X_i > 2] \lesssim
	n^{-4}$, concluding the proof of~\prettyref{eq:emp-subg1}.
	
	Next recall that for a given class of priors $\calG$, the total regret in the compound setting is defined as
	\[
\TotRegretComp_n(\calG) \eqdef \inf_{\hat \theta^n} \sup_{\theta^n \in \Theta(\calG)} \sth{ \EE_{\theta^n}[\|\hat
\theta^n(Y^n) - \theta^n\|^2] - n \cdot \mmse(G_{\theta^n})}\,,
\]
where $\Theta(\calG) \triangleq \{\theta^n: G_{\theta^n} \in \calG\}$.
A simple observation is that for $\TotRegretComp_n(\SubG(s))$, it is sufficient to restrict to $\hat \theta^n = \hat \theta^n(Y^n)$ such that $\|\hat \theta^n\| \leq \sqrt{C_2 s n}$ deterministically for some absolute constant $C_2$. This simply follows from the fact that for any $\theta^n\in\Theta(\SubG(s))$, $\|\theta^n\|\leq \sqrt{C_2 s n}$, so that for any estimator $\hat \theta^n$, the modification 
$\tilde \theta^n \triangleq \argmin\{ \|\hat \theta^n - x\|: \|x\|\leq \sqrt{C_2 s n}\}$ improves $\hat\theta^n$ \emph{pointwise}, in the sense that $\|\tilde \theta^n -\theta^n\| \leq \|\hat\theta^n -\theta^n\|$ for every  $\theta^n\in\Theta(\SubG(s))$.

Finally, we show
\begin{equation}\label{eq:ebrt2}
\TotRegretComp_n(\SubG(C_0 s)) \geq \TotRegret_n(\SubG(s))  - O(s n^{-1})
\end{equation}
To this end, fix any estimator $\hat\theta^n$ such that  $\|\hat \theta^n\| \leq \sqrt{C_2 C_0 s n}$. 
For any $G \in \SubG(s)$, let $\theta_1,\ldots,\theta_n\iiddistr G$. Then
\begin{align}
& ~  \sup_{\theta^n \in \Theta(\SubG(C_0 s))} \sth{ \EE_{\theta^n}[\|\hat
\theta^n(Y^n) - \theta^n\|^2] - n \cdot \mmse(G_{\theta^n})} \nonumber\\
\geq & ~  \EE\qth{\|\hat \theta^n - \theta^n\|^2 - n \cdot \mmse(G_{\theta^n}) \mid G_{\theta^n}
\in \SubG(C_0 s) } \nonumber\\
\geq & ~ \EE[\|\hat \theta^n - \theta^n\|^2 - n \cdot \mmse(G_{\theta^n})] - \EE[\|\hat \theta^n -
\theta^n\|^2 \indc{ G_{\theta^n} \notin \SubG(C_0 s)} ] \nonumber\\
\stepa{\geq} & ~ \EE[\|\hat \theta^n - \theta^n\|^2] - n \cdot \mmse(G) - \sqrt{\EE[\|\hat
\theta^n - \theta^n\|^4] \prob{ G_{\theta^n} \notin \SubG(C_0 s)}}  \nonumber\\
\stepb{\geq} & ~ \EE[\|\hat \theta^n - \theta^n\|^2] - n \cdot \mmse(G) - O({s\over
n}),\label{eq:ebrt}
\end{align}
where in (a) we apply \prettyref{eq:mmse-emp} again for the first term and Cauchy-Schwarz for the second; 
(b) follows from \prettyref{eq:emp-subg} and the following bounds
\begin{align*} \EE[\|\hat \theta^n - \theta^n\|^4] &\le 8 (\EE[\|\hat \theta^n\|^4] + \EE[\|\theta^n\|^4]\\
		&\stepc{\le} 8((C_2C_0sn)^2 + n^2 \EE^2[\theta^2] + n \Var[\theta^2])\\
		&\lesssim n^2s^2
\end{align*}
and in (c) we used the fact that $\|\hat\theta^n\|\le \sqrt{C_2C_0 sn }$.
Taking the infimum on both sides of~\eqref{eq:ebrt} over $\hat \theta^n$ subject to the norm
constraint followed by a supremum over $G \in \SubG(s)$ over the right side, we complete the proof
of~\eqref{eq:ebrt2}. Since $\TotRegret(\SubG(c_0s))$ is increasing with $c_0$, the
statement~\eqref{eq:ebrt2} is equivalent to that of the Theorem upon taking $c_0={1\over KC_0}$
with a suitably large $K$.
\end{proof}

\section{Results on density estimation}
\label{app:density}

The construction in the present paper can be used to yield improved or new lower bounds for mixture density estimation.
Denote by the minimax squared Hellinger risk for density estimation over the mixture class $\{f_{G}: G\in\calG\}$, namely,
	\[
	R_n(\calG) \triangleq \inf_{\hat f}\sup_{G\in\calG} \Expect_G[H^2(f_{G},\hat f)],
	\]
	where $\hat f$ is measurable with respect to the sample $Y_1,\ldots,Y_n \iiddistr f_G$ and 
	the Hellinger distance between densities $f$ and $g$ with respect to a dominating measure $\nu$ is denoted by 
	$H(f,g) \triangleq (\int (\sqrt{f}-\sqrt{g})^2 d\nu)^{1/2}$.

\paragraph{General program.}
The following result is a counterpart of \prettyref{prop:ass_gen}.
Since we no longer dealing with regression functions, the program only involves the operator $K$ in 
\prettyref{eq:k_def} not the higher-order operator $K_1$.

\begin{proposition}\label{prop:ass_gen-dens} Fix a prior distribution $G_0$, constants $a,\tau,\tau_2,\gamma \ge 0$
and $m$ functions $r_1,\ldots,r_m$ on $\calX$ with the following properties.
\begin{enumerate}
\item For each $i$, $ \|r_i\|_{\infty} \le a $;

\item For each $i$, $\|K r_i\|_{L_2(f_0)} \le \sqrt{\gamma}$;

\item For any $v \in \{0, \pm 1\}^m$,
	\begin{equation}\label{eq:rg_ag-dens}
		\left\|\sum_{i=1}^m v_i K r_i\right\|^2_{L_2(f_0)} \ge \tau \|v\|_2^2 - \tau_2;
\end{equation}	

\item For each $i$, $\int r_i dG_0=0$.
\apxonly{If we don't assume this, we need to replace \prettyref{eq:rg_ag-dens} $h_i=K r_i$ with its centered version $\bar h_i = h_i-\mu_i$, namely, 
$\|\bar h_v\right\|^2_{L_2(f_0)} \ge \tau \|v\|_2^2 - \tau_2$.}
\end{enumerate}
Then the minimax $H^2$-risk for density estimation over the class of priors $\calG = \{G: \left|{dG/dG_0} - 1\right| \le {1\over 2}\}$ satisfies
\begin{equation}\label{eq:propass_g-dens}
	R_n(\calG) \ge C \delta^2 (m \tau - \tau_2) \,, \qquad \delta \triangleq \frac{1}{\max(\sqrt{n \gamma},
ma)}\,,
\end{equation}
where $C>0$ is an absolute constant.
\end{proposition}

\begin{proof}
	The proof is almost identical to that of \prettyref{prop:ass_gen} based on Assouad's lemma, so we only point out the major differences.
	By the fourth assumption, $\mu_i = \int r_i dG_0=0$, so we simply have
	$dG_u = (1+\delta r_u) dG_0$ and $f_u = (1+\delta h_u) f_0$, where $h_u = K r_u$.
	By \prettyref{eq:rg_c1g}, we have $\frac{1}{2}\leq f_u/f_0 \leq \frac{3}{2}$, so for any $u,v\in\{0,1\}^m$,
	\[
	H^2(f_u,f_v) \asymp \int \frac{(f_u-f_v)^2}{f_0} = \delta^2 \|K(r_u-r_v)\|_{L_2(f_0)}^2 \geq \delta^2 (\tau d_{\rm H}(u,v)-\tau_2).
	\]
	Then the conclusion follows from applying Assouad's lemma as in the proof of \prettyref{prop:ass_gen}.
\end{proof}

\paragraph{Truncation.}

\prettyref{lmm:truncation} controls the effect of truncation on the regret. The counterpart for density estimation is as follows. 
\begin{lemma}
\label{lmm:truncation-density}	
Under the setting of \prettyref{lmm:truncation}, 
	\begin{equation}
	R_n(\calG) \geq  R_n(\calG') - 8 n\sqrt{\epsilon}.
	\label{eq:truncate-density}
	\end{equation}
\end{lemma}
\begin{proof}
	Indeed, using $G_a$ as defined in the proof of Lemma~\ref{lmm:truncation},
		\begin{align*}
	R_n(\calG)
	= & ~ \inf_{\hat f} \sup_{G \in \calG} \Expect_G [H^2(\hat f, f_G)]  \\
	\geq & ~ \inf_{\hat f} \sup_{G \in \calG'} \Expect_{G_a} [H^2(\hat f, f_{G_a})]  \\
	\stepa{\geq} & ~ \inf_{\hat f} \sup_{G \in \calG'} \Expect_{G} [H^2(\hat f, f_{G_a})]	- 2n\epsilon  \\
	\stepb{\geq} & ~ \inf_{\hat f} \sup_{G \in \calG'} \Expect_{G} [H^2(\hat f, f_{G})] - 6 H(f_G, f_{G_a})	-2 n\epsilon  \\
	\stepc{\geq} & ~  R_n(\calG') - 8n \sqrt{\epsilon},
	\end{align*}
		where (a) follows from the following counterpart of \eqref{eq:tr_1} since
		$H^2 \le 2$:
		$\Expect_G[H^2(f_G,\hat f)] \leq \Expect_G[H^2(f_G,\hat f)|E] + 2 n \epsilon = \Expect_{G_a}[H^2(f_G,\hat f)] + 2 n \epsilon$;
	(b) uses the triangle inequality and boundedness of Hellinger distance;
	(c) is due to $H^2(f_G, f_{G_a}) \leq H^2(G,G_a) \leq \TV(G,G_a) = \epsilon$, with the first step via the data processing inequality.
\end{proof}

\paragraph{Applications}
As a concrete application, we reuse the construction in \prettyref{sec:pf-gau} and \prettyref{apx:normal} to derive a minimax lower bound next for estimating Gaussian mixture densities. 
In particular, the lower bound $\Omega(\frac{\log n}{n})$ for $s$-subgaussian priors is previously shown in \cite{kim2014minimax,kim2020minimax} for sufficiently large $s$, an assumption which in fact can be dropped; the lower bound $\Omega(\frac{\log n}{n \log\log n})$ for bounded means appears to be new.

\begin{theorem}[Gaussian mixture estimation]
\label{thm:gaussian-density} 
Consider the setting of \prettyref{th:gaussian}.
\begin{itemize}
	\item 
	(Compactly supported case). 
For any $h>0$, there exists constants $c_0=c_0(h)>0$ and $n_0$, such that for all $n\geq n_0$,
	\begin{equation}\label{eq:gaussian-density1}
		R_n(\calP([-h,h]) \ge \frac{c_0}{n} {\log n \over \log \log n}\,.
	\end{equation}		
	
	\item 
	(Subgaussian case.) 
	For any $s>0$, there exists constants $c_1=c_1(s)>0$ and $n_0$, such that for all $n\ge n_0$,
	\begin{equation}\label{eq:gaussian-density2}
		R_n(\SubG(s)) \ge \frac{c_1 }{n} \log n\,.
	\end{equation}

	\end{itemize}
\end{theorem}
\begin{proof}
	We show 
	\begin{equation}
	R_n(\SubG(s)) \geq \frac{c \log n}{n}
	\label{eq:gaussian-density3}
	\end{equation}
	for some absolute constant $c$. Taking $s = \frac{c(h)}{\log n}$ and applying the truncation argument in \prettyref{eq:truncate-density} yields \prettyref{eq:gaussian-density1} for the compactly supported case.
	
	To prove \prettyref{eq:gaussian-density3}, we apply \prettyref{prop:ass_gen-dens} with 
	$r_k  = \frac{\psi_k}{\|K\psi_k\|_{L_2(f_0)}}, k=1,\ldots,m$, where $\{\psi_k\}$ is the class of functions from \prettyref{lem:normal_main}, so that $\tau=\gamma=1, \tau_2=0$ and $a= \sqrt{\frac{\alpha_1}{\lambda_0\mu^m}}$.
	Since $\mu \asymp s$, choosing $m = \frac{c' \log n}{\log \frac{1}{s}}$ for small constant $c'$ ensures $n \gamma \geq (ma)^2$, leading to the lower bound of $\Omega(\frac{m}{n})$.
\end{proof}

The counterpart for the Poisson model is as follows. The proof is almost identical to that of \prettyref{thm:gaussian-density} using the construction in \prettyref{lem:poisson_main} and \prettyref{lem:poisson_exp}, and is hence omitted.
It is not hard to show that the lower bounds in \prettyref{eq:poisson-density1} and \prettyref{eq:poisson-density2}  are in fact optimal \cite{JPW21}.

\begin{theorem}[Poissoin mixture estimation]
\label{thm:poisson-density} 
Consider the setting of \prettyref{th:poisson}.
\begin{itemize}
	\item 
	(Compactly supported case.)
For any $h>0$, there exists constants $c_0=c_0(h)>0$ and $n_0$, such that for all $n\geq n_0$,
	\begin{equation}\label{eq:poisson-density1}
		R_n(\calP([0,h]) \ge \frac{c_0}{n} {\log n \over \log \log n}\,.
	\end{equation}		
	
	\item 
	(Subexponential case.) 
	For any $s>0$, there exists constants $c_1=c_1(s)>0$ and $n_0$, such that for all $n\ge n_0$,
	\begin{equation}\label{eq:poisson-density2}
		R_n(\SubE(s)) \ge \frac{c_1 }{n} \log n\,.
	\end{equation}

	\end{itemize}
\end{theorem}
\apxonly{
\begin{proof}
	Again we take $G_0=\GammaD(\alpha,\beta)$. Take 
	\[
	r_k = \frac{\Gamma_k}{\|K \Gamma_k\|_{L_2(f_0)}} = \underbrace{\frac{\Gamma_k}{\|K_1 \Gamma_k\|_{L_2(f_0)}}}_{\text{previous choice for \prettyref{lem:poisson_main}}} \underbrace{\sqrt{\frac{(S_1\Gamma_k,\Gamma_k)}{(S\Gamma_k,\Gamma_k)}}}_{\text{poly}(m)}
	\]
	Applying \prettyref{prop:ass_gen-dens} with $\tau=\gamma=1$ and $\tau_2=0$ so the end result is $\Omega(\frac{m}{n})$.
\end{proof}
}


\end{document}